\documentclass[11pt,reqno]{amsart}

\usepackage[utf8]{inputenc} 

\usepackage[margin=1in]{geometry} 

\usepackage{graphicx} 
\usepackage{float} 

\usepackage{changepage}
\usepackage{stackengine}

 \usepackage[parfill]{parskip} 
 
\usepackage{booktabs} 
\usepackage{array} 
\usepackage{paralist} 
\usepackage{verbatim} 
\usepackage{subfig} 
\usepackage{mathrsfs}
\usepackage{amssymb}
\usepackage{xcolor}
\usepackage{amsthm}
\usepackage{amsmath,amsfonts,amssymb,esint,hyperref, varioref}
\usepackage[noabbrev, capitalize]{cleveref}

\usepackage{autonum}

\usepackage{graphics,color}
\usepackage{enumerate, enumitem}
\usepackage{mathtools,centernot}
\usepackage{cases}
\usepackage{amsrefs}
\usepackage{bbm}
\usepackage{xfrac}

\usepackage{bookmark}

\newtheorem{theorem}{Theorem}[section]
\newtheorem{lemma}[theorem]{Lemma}

\newtheorem{corollary}[theorem]{Corollary}

\newtheorem{proposition}[theorem]{Proposition}
\newtheorem{remark}[theorem]{Remark}

\numberwithin{equation}{section} 

\newcommand{\norm}[1]{\left\|#1\right\|}
\newcommand{\abs}[1]{\left|#1\right|}

\newcommand{\T}{\ensuremath{\mathbb{T}}}
\newcommand*{\R}{\ensuremath{\mathbb{R}}}

\newcommand*{\Z}{\ensuremath{\mathbb{Z}}}

\newcommand{\eps}{\varepsilon}

\newcommand{\quotes}[1]{``#1''}

\renewcommand{\MR}[1]{} 

\usepackage{color, graphicx}
\usepackage{mathrsfs, dsfont}

\usepackage[]{hyperref}
\hypersetup{
    colorlinks=true,       
    linkcolor=red,          
    citecolor=blue,        
    filecolor=red,      
    urlcolor=cyan           
}

\def\dist{\mathop{\rm dist}\nolimits}    
\def\div{\mathop{\rm div}\nolimits}    
 
\def\dim{\mathop{\rm dim}\nolimits}
\def\spt{\mathop{\rm Spt}\nolimits}

\newcommand{\be}{\begin{equation}}
\newcommand{\ee}{\end{equation}}

\title{Intermittency and Dissipation Regularity in Turbulence}

\author[L. De Rosa]{Luigi De Rosa}
\address[L. De Rosa]{Gran Sasso Science Institute, viale Francesco Crispi, 7, 67100 L’Aquila, Italy}
\email{luigi.derosa@gssi.it}

\author[T. D. Drivas]{Theodore D. Drivas}
\address[T. D. Drivas]{Department of Mathematics, Stony Brook University, Stony Brook, NY, 11794, USA}
\email{tdrivas@math.stonybrook.edu}

\author[M. Inversi]{Marco Inversi}
\address[M. Inversi]{Departement Mathematik und Informatik, Universit\"at Basel, CH-4051 Basel, Switzerland}
\email{marco.inversi@unibas.ch}

\author[P. Isett]{Philip Isett}
\address[P. Isett]{Department of Mathematics, California Institute of Technology, Pasadena CA-91125, USA}
\email{isett@caltech.edu}

\date{\today}


\begin{document}

\begin{abstract}
We lay down a geometric-analytic framework to capture properties of  energy dissipation within weak solutions to the incompressible Euler equations. For solutions with spatial Besov regularity, it is proved that the Duchon--Robert distribution has optimal improved regularity in a negative Besov space and, in the case it is a Radon measure, it is absolutely continuous with respect to a suitable Hausdorff measure. This imposes quantitative constraints on the dimension of the, possibly fractal, dissipative set and the admissible structure functions exponents, relating to the phenomenon of ``intermittency'' in turbulence. 
As a by-product of the approach, we also recover many  known ``Onsager singularity'' type results.
\end{abstract}

\maketitle
\section{Introduction}
We consider the incompressible Euler equations
\begin{equation}\label{E} \tag{E}
\left\{\begin{array}{l}
\partial_t u +\div (u \otimes u) +\nabla q =0 \\
\div u=0
\end{array}\right. \qquad \text{in }\T^d \times (0,T).
\end{equation}  
By the work of Duchon--Robert \cite{DR00},  weak solutions of  \eqref{E} in $L^3_{x,t}$ satisfy a local energy  balance
\begin{equation}
    \label{enbal_E}
\partial_t \frac{\abs{u}^2}{2} +\div  \left(\left(\frac{\abs{u}^2}{2}+q\right)u\right)=-D  \qquad \text{in } \mathcal D'_{x,t},
\end{equation}
with a distributional source/sink term $D$, known as the ``Duchon--Robert distribution", representing the effect of anomalous dissipation, a central phenomenon in the context of ``fully developed turbulence" and the related Kolmogorov \cite{K41} and Onsager \cite{O49} theories.  Motivated by the ubiquity of ``intermittency" in turbulent fluids, we are concerned with geometric/analytic properties of the Duchon--Robert distribution under local regularity assumptions on the weak solution.  Our main theorem is the following.

\begin{theorem}[Dissipation regularity]
    \label{T:main D regularity}
    Assume that $u\in L^p_t B^\sigma_{p,\infty}$ is a weak solution to \eqref{E} for some   $p\in [3,\infty]$ and $\sigma \in \left(0,1\right)$, with Duchon--Robert distribution $D$. Then $D\in B^{\frac{2\sigma}{1-\sigma}-1}_{\frac{p}{3},\infty}$ locally in space-time. If in addition $D$ is a real-valued Radon measure, we have 
     \begin{itemize}
            \item[(i)] $\abs{D}$  is absolutely continuous with respect to $\mathcal{ H}^\gamma$ for any $\gamma \geq 0$ such that
            \begin{equation}\label{gamma_condition}
                \frac{2\sigma}{1-\sigma}>1- \frac{p-3}{p} (d+1-\gamma).
            \end{equation}
            \item[(ii)]  If $D\geq 0$,  $\forall K$ compact $\exists r_0>0$ such that
            \begin{equation}\label{density bound}
                D(B_r (x,t))\lesssim r^{\frac{2\sigma}{1-\sigma}-1 +\frac{p-3}{p}(d+1)}\qquad \forall (x,t)\in K, \, \forall r<r_0.
                \end{equation}
    \end{itemize}
\end{theorem}
In the theorem there is no reference to the pressure since weak solutions to \eqref{E} can be defined by testing with divergence-free vector fields. The pressure is then recovered a posteriori as the unique zero-mean solution to 
$$
-\Delta q = \div \div (u\otimes u).
$$

The assumption that $D$ is a Radon measure is satisfied, for instance, whenever $u$ arises as a strong $L^3_{x,t}$ limit of suitable Navier--Stokes weak solutions. See Section \ref{S:intermittency turbulence} for further discussion. This theorem generalizes the results of \cite{DRIS24}. More importantly, it lays down a framework which captures fine properties on possibly densely distributed concentration sets of $D$. Indeed, the arguments from \cite{DRIS24} (after \cites{Is24,DRH23}) fail to give any non-trivial conclusion as soon as the dissipative set is anywhere dense. To overcome this difficulty a more intrinsic approach is necessary.  {The use of the Hausdorff measure $\mathcal H^\gamma$ to quantify the regularity of $D$ has been made for convenience\footnote{{Treating space and time equally makes the results cleaner. However, because of the space-time anisotropic nature of \eqref{enbal_E}, it is very unclear whether a universally better choice exists.}}, but different choices would lead to different conclusions. When $D$ is only a distribution, a property similar to $(i)$ can still be deduced (see Remark \ref{R:distrib support}).} The uniform regularity of $D$ on space-time balls in $(ii)$ provides a $L^\infty_{x,t}$ bound on the upper fractional density of $D$ (see Remark \ref{R:density bound}) which might be relevant in view of the \quotes{multifractal} nature of turbulence \cites{BMV08,FP85,BPPV84,Falc}. The argument of Theorem \ref{T:main D regularity} is completely local and it carries over to the case of general open sets modulo minor technical details (see Section \ref{S:open_set}).

Theorem \ref{T:main D regularity} is relevant in the range $\sigma \in \left( 0 ,\frac13\right]$. In fact, the conclusions are slightly suboptimal at the endpoint $\sigma =\frac13$ (see Remark \ref{R:endpoint}). For any $\sigma\in \left( 0, \frac13\right)$, the negative Besov regularity  $D\in  B^{\frac{2\sigma}{1-\sigma}-1}_{\frac{p}{3},\infty}$ is optimal (see Theorem \ref{T:sharp_baire}). Moreover, for a general measure in a negative Besov space, the absolute continuity with respect to $\mathcal H^\gamma$ cannot be improved (see Remark \ref{R:negative measures sharp}). It follows that Theorem \ref{T:main D regularity} gives quite general, and sharp in the appropriate sense, rigidity results on the incompressible Euler equations\footnote{{Since the possibility for the inviscid limit to impose additional geometric properties on $D$ is very unclear, the case in which the Euler solution arises as a physical realization might, in principle, be quite different.}}. They establish the positive side of an intermittent version of the Onsager conjecture, which is a natural extension 
in view of the breakdown of self-similarity and space-time homogeneity in real turbulent flows. Indeed, a direct consequence is the following intermittency-type statement.

 \begin{corollary}[Intermittency]\label{C:main intermittency}
 Let $u\in L^3_{x,t}$ be a weak solution to \eqref{E}. Assume that the Duchon--Robert distribution is a measure whose singular part with respect to the Lebesgue measure is  non-trivial and concentrated on a space-time set $S$ with $\dim_\mathcal{H} S = \gamma$. 
For all  $p\in [3,\infty]$ for which there exists $\sigma_p\in (0,1)$ such that $u\in L^p_t B^{\sigma_p}_{p,\infty}$, it must hold
    \begin{equation}\label{eq:intermittency}
         \frac{2\sigma_p}{1-\sigma_p}\leq 1- \frac{p-3}{p} (d+1-\gamma).
            \end{equation}
\end{corollary}

Note that as long as the dimension $\gamma$ is less than $d + 1$, then for $p > 3$ this bound implies the regularity index $\sigma_p$ must be strictly below $\frac13$. {In particular, the presence of any non-trivial lower-dimensional dissipation would necessarily result in a quantitative downward deviation from the Besov $\frac13$ regularity for all $p>3$, translating into the failure of the Kolmogorov prediction of linear structure functions exponents (see Section \ref{S:intermittency turbulence} for elaboration). As no assumption was made on the Lebesgue regular part of $D$, Corollary \ref{C:main intermittency} applies to the setting in which the dissipation has a full-dimensional piece as well. This improves over the existing literature and, given the difficulty in excluding full-dimensional parts experimentally, it has a wider and perhaps more realistic  range of applications.}

Theorem \ref{T:main D regularity} is a consequence of the following energy identity for $L^3_{x,t}$ weak solutions.

\begin{proposition}[Modified energy identity]\label{P:en_ident}
 Let $u\in L^3_{x,t}$ be a weak solution to \eqref{E} with Duchon--Robert distribution $D$. Let
     \begin{align}
     E^\ell&:=\frac{\abs{u-u_\ell}^2}{2},\\ 
     Q^\ell &:= \left(\frac{\abs{u-u_\ell}^2}{2} + (q-q_\ell) \right) (u-u_\ell), \\
     R^\ell&:= u_\ell\otimes u_\ell - (u\otimes u)_\ell,\\
     C^\ell &:= (u-u_\ell) \cdot \div R^\ell + (u-u_\ell)\otimes (u-u_\ell):\nabla u_\ell,
     \end{align}
     with $u_\ell$ the space mollification of $u$.
    For all $\ell>0$,  the following identity holds
     \begin{equation}
         \label{D_decomp}
         -D= (\partial_t+u_\ell \cdot \nabla) E^\ell +  \div Q^\ell + C^\ell \qquad \text{in } \mathcal D'_{x,t}.
     \end{equation}
\end{proposition}
In Proposition \ref{P:decomposition_NS} the above identity is proved in the more general case of the Navier--Stokes equations. These identities split the dissipation into terms that are small in negative norms, i.e. the ones with $E^\ell$ and $Q^\ell$, and a term that is large in a positive norm, i.e. $C^\ell$. By optimizing in the choice of $\ell$, we are able to deduce quantitative rates when approximating $D$ with its space-time mollification.

\begin{proposition}[Mollification rates]\label{P:quant moll est}
     Assume that $u\in L^p_t B^\sigma_{p,\infty}$ is a weak solution to \eqref{E} for some   $p\in [3,\infty]$ and $\sigma \in \left(0,1\right)$, with Duchon--Robert distribution $D$. Let $\rho_\delta$ be a space-time Friedrichs' mollifier. For any $\varphi\in C^\infty_{x,t}$ with compact support there exists $\delta_0>0$, which depends only on the distance of the support of $\varphi$ from the boundary of the space-time domain\footnote{This is only used to guarantee that $\langle D*\rho_\delta,\varphi \rangle$ is well-defined.}, such that 
        \begin{equation}\label{moll_est_dissipation}
            \abs{\left\langle   D-D* \rho_\delta,\varphi\right\rangle   } \lesssim \delta^{\frac{2\sigma}{1-\sigma}} \norm{\varphi}_{W^{1,\frac{p}{p-3}}_{x,t}} \qquad \text{and} \qquad  \abs{\left\langle   D* \rho_\delta,\varphi\right\rangle   } \lesssim \delta^{\frac{2\sigma}{1-\sigma}-1} \norm{\varphi}_{L^{\frac{p}{p-3}}_{x,t}}
        \end{equation}
        for all $\delta<\delta_0$.
\end{proposition}

When cutting $D$ into frequency shells, the two estimates in \eqref{moll_est_dissipation} can be used to control all the pieces in the Fourier space, from which the negative Besov regularity for $D$ is deduced. In a different but equivalent terminology, \eqref{moll_est_dissipation} can be read by duality as a quantitative convergence of $D* \rho_\delta$ to $D$ in a negative Sobolev norm and a controlled blowup of $D* \rho_\delta$ in $L^{\frac{p}{3}}_{x,t}$ respectively. Then, the abstract linear interpolation \cites{BL76,Lun09} leads to the desired negative fractional regularity. 

Although the main purpose of the identity \eqref{D_decomp} in this note, together with its analogue \eqref{D_decomp_NS} for the Navier--Stokes equations, is to prepare the ground for Theorem \ref{T:main D regularity}, in Section \ref{s: corollaries} we show how all the main energy-type results in this context follow as almost immediate corollaries. Some of them are well known, some are improved versions of previous results, some others are new.  We remark also that the results of this paper could be extended to any system of conservation laws, including the transport equation, a case which we outline in Section \ref{transport}, but also systems like magnetohydrodynamics \cites{aluie2010scale,eyink2006breakdown} or compressible fluids \cites{drivas2018onsager,bardos2019onsager}.  

In Section \ref{s: tools} we list the main tools that will be used in this note, in Section \ref{s: proofs} we prove our main results while in Section \ref{s: corollaries} all the corollaries. Finally, Section \ref{s: comments} is dedicated to discussions. These include: the main theoretical background, physical significance of the statements, comparison with previous related works, intermittency for the linear transport equation, sharpness of results and their link with the available convex-integration constructions.


\section{Tools} \label{s: tools} In this section we recall the main tools used in this note. 

\subsection{Besov spaces} 
We define the Besov spaces on $\R^N$ by means of the Littlewood--Paley decomposition (see e.g. \cite{BL76}*{Chapter 6}). Let $\phi =\phi(\xi)$ be a smooth function such that 
$$
\spt\phi\subset \left\{\xi\in \R^N \,:\, \frac{1}{2}<\abs{\xi}<2 \right\} \qquad \text{and} \qquad \sum_{k\in \Z}  \phi\left( 2^{-k} \xi\right)=1 \quad \forall \xi\in \R^N\setminus \{0\}.
$$

For any $k\in \Z$ we define $\phi_k,\psi \in C^\infty$ such that
\begin{align} \label{phi_k}
    \hat{\phi}_k(\xi)&:=\phi\left( 2^{-k}\xi\right) \qquad \text{and} \qquad \hat{\psi}(\xi) :=1- \sum_{k\geq 1} \phi\left( 2^{-k}\xi\right),
\end{align}
where the symbol $\hat{\cdot}$ denotes the Fourier transform.
Since $\hat{\psi}\in C^\infty$ and $\spt \hat \psi\subset B_2(0)$, then $\psi\in \mathcal S$, the space of Schwartz functions. Consequently, for any $p\in [1,\infty]$, any $\alpha\in \R$ and any tempered distribution $f$, we define the Besov norm on $\R^N$ by
\begin{equation}
    \label{besov_norm}
    \norm{f}_{B^\alpha_{p,\infty}}:=\norm{f*\psi}_{L^p} + \sup_{k\geq 1}\left( 2^{k\alpha} \norm{ f*\phi_k}_{L^p}\right).
\end{equation}

If $U \subset \R^N$ is an open set, and not necessarily the whole space, we say that $f$ belongs to $B^\alpha_{p, \infty}$ \quotes{locally inside $U$} if  $ \chi f  \in B^\alpha_{p, \infty}$ on the whole space for any smooth $\chi$ with compact support in $U$.

\subsection{Mollification estimates}
We recall some classical mollification estimates. For the proof see for instance \cites{drivas2026mathematical,YWW25}.

\begin{lemma}
For any function $f: \R^N \to \R$ denote by $f_\ell = f* \rho_\ell$, where $\rho$ is a Friedrichs' mollifier. Fix $\sigma, \alpha \in (0,1)$ and $p \in [1, \infty]$. There exist implicit constants independent of $\ell$ such that
\begin{align}
    \norm{f-f_\ell}_{L^p}\lesssim &  \ \ell^\sigma \norm{ f}_{B^\sigma_{p,\infty}},\label{moll_est_1}\\
    \norm{\nabla^n f_\ell}_{L^p}\lesssim & \ \ell^{\sigma-n} \norm{f}_{B^\sigma_{p,\infty}}\qquad\qquad\qquad\qquad n\geq 1,\label{moll_est_2}\\
    \norm{\nabla^n (f_\ell g_\ell- (fg)_\ell)}_{L^{p}}\lesssim& \  \ell^{\sigma+ \alpha-n}\norm{ f}_{B^\sigma_{rp,\infty}} \norm{ g}_{B^\alpha_{r'p,\infty}} \qquad\  n\geq 0, \ \ \ \,\frac{1}{r}+\frac{1}{r'}=1\label{moll_est_3}.
\end{align}
\end{lemma}

From \cites{CDF20,Isett23,ColDeRos} it is known that the pressure enjoys the double regularity
\begin{equation}\label{pressure_double}
    \|q\|_{L^{\frac{p}{2}}_t B^{2\sigma}_{\frac{p}{2},\infty}}\lesssim \|u\|^2_{L^p_t B^\sigma_{p,\infty}}\
\end{equation}
for any $p\in (2,\infty]$ and $\sigma \in ( 0,1)$. Then, the following is a  consequence of \eqref{moll_est_1}, \eqref{moll_est_2}, \eqref{moll_est_3} and \eqref{pressure_double}.

\begin{corollary} 
Let $p\in [3,\infty]$ and $\sigma \in (0,1)$. Let $E^\ell$, $Q^\ell$, $R^\ell, C^\ell$ be the quantities defined in the statement of Proposition \ref{P:en_ident}. There exist implicit constants independent of $\ell$ such that
\begin{align}
    \norm{E^\ell}_{L^\frac{p}{2}_{x,t}}&\lesssim \ell^{2\sigma} \norm{u}^2_{L^p_t B^\sigma_{p,\infty}} \label{est_Eell}, \\
     \norm{Q^\ell}_{L^\frac{p}{3}_{x,t}}&\lesssim \ell^{3\sigma} \norm{u}^3_{L^p_t B^\sigma_{p,\infty}}\label{est_Qell}, \\
         \norm{R^\ell}_{L^\frac{p}{2}_{x,t}} + \ell  \norm{\div R^\ell}_{L^\frac{p}{2}_{x,t}}&\lesssim \ell^{2\sigma} \norm{u}^2_{L^p_t B^\sigma_{p,\infty}}\label{est_Rell}, \\
         \norm{C^\ell}_{L^\frac{p}{3}_{x,t}}&\lesssim \ell^{3\sigma-1} \norm{u}^3_{L^p_t B^\sigma_{p,\infty}}\label{est_Cell}.
\end{align}
\end{corollary}

\subsection{Radon measures}
We follow the presentation from \cite{Magg12}. Let $U\subset \R^N$ be open. A real-valued Radon measure $\mu$ on $U$ is a bounded linear functional over the space of continuous functions with compact support in $U$. The variation measure of $\mu$, denoted by $|\mu|$, is the non-negative measure such that 
$$
\langle |\mu|,\varphi\rangle = \int \varphi \, d|\mu| :=\sup_{g\in C^0, \, |g|\leq \varphi} \langle \mu, g \rangle \qquad \forall \varphi \in C^0_c,\,\varphi \geq 0.
$$
Clearly, $|\mu|=\mu$ for any non-negative $\mu$. By the Riesz theorem, a Radon measure $\mu$ can be identified with a countably additive set function defined over the Borel subsets of $U$, and $|\mu|(K)<\infty$ for any compact set $K\subset U$. Radon measures are inner and outer regular, meaning that
$$
|\mu|(A) = \sup\{|\mu| (K) \, : \, K\subset A, \, K \text{ compact}  \}
$$
and 
$$
|\mu|(A) =\inf \{|\mu| (O) \, : \, A\subset O, \, O \text{ open}   \}
$$
hold for any Borel set $A\subset U$.

Given a Borel set $S\subset U$ we denote by $\mu\llcorner S$ the measure defined as 
$$
\mu\llcorner S(A):= \mu(A\cap S) \qquad \forall A\subset U, \, A \text{ Borel.}
$$
We say that $\mu$ is concentrated on $S$ if $\mu \equiv \mu\llcorner S$. Consequently, two measure $\mu_1$ and $\mu_2$ are said to be singular with respect to each other if they are concentrated on disjoint sets, i.e. $\mu_1\equiv \mu_1\llcorner S_1$ and $\mu_1\equiv \mu_1\llcorner S_2$ for some sets $S_1,S_2$ such that $S_1\cap S_2=\emptyset$.

The support of a distribution $L$, denoted by $\spt L$, is the complement of the union of all the open sets $O$ such that $\langle L,\varphi\rangle =0$ for all test functions $\varphi$ with compact support in $O$.  It follows that any measure is concentrated on its support. By the Hahn decomposition theorem, for any real-valued measure $\mu$ we find two disjoint Borel sets $S_+,S_-\subset U$ such that $\mu\llcorner S_{\pm}\geq 0$ and $S_+\cup S_-=U$. Consequently, $\mu=\mu\llcorner S_{+}-\mu\llcorner S_{-}$ is a decomposition of $\mu$ into a positive and a negative part. 
It holds $|\mu|= \mu\llcorner S_{+}+\mu\llcorner S_{-}$ as measures. In particular, $\mu$ is concentrated on $S$ if and only if $|\mu| (S^c)=0$. 

Let $\mu$ and $\lambda$ be two Borel measures, $\lambda$ non-negative. We say that $\mu$ is absolutely continuous with respect to $\lambda$, written as $\mu\ll \lambda$, if $|\mu| (A)=0$ for all Borel sets $A$ such that $\lambda(A)=0$. This is equivalent to $|\mu|\ll \lambda$.

\subsection{Functional inequalities}

We recall the following.

\begin{lemma}[Shinbrot \cite{S74}*{Lemma 4.2}]\label{L:shin}
Let $v$ be a time dependent vector field and $f,g$ be two functions in space-time. Let $\frac{2}{p}+\frac{2}{m}=1$ with $p\geq 4$. Then 
$$ \abs{\int g v\cdot \nabla f} \leq \norm{v}_{L^m_t L^p_x} \norm{\nabla f}_{L^2_{x,t}} \norm{ g}_{L^\infty_tL^2_x}^{2-\frac{m}{2}}\norm{g}_{L^m_t L^p_x}^{\frac{m}{2}-1}.
$$
\end{lemma}

We will need a Young's inequality for convolutions in negative spaces. For any $p\in [1,\infty]$, we denote by $W^{-1,p}$ the dual of $W^{1,p'}$ with $\frac{1}{p}+\frac{1}{p'}=1$.

\begin{lemma}\label{L:young}
On $\R^N$, let $L\in \mathcal D'$ be with compact support and $\phi\in \mathcal S$. For any $p\in [1,\infty]$ it holds
$$
\norm{ L*\phi}_{L^p} \leq \norm{L}_{W^{-1,p}} \norm{\phi}_{W^{1,1}}.
$$
\end{lemma}
\begin{proof}
Since $L$ has compact support, then $L*\phi\in \mathcal S$. Let $\varphi \in C^\infty_c$ be arbitrary and $p'$ be the H\"older conjugate of $p$. Denote by $\tilde{\phi}(x) = \phi(-x)$. The standard Young's convolution inequality implies
$$ \abs{ \int (L*\phi) \, \varphi}=\abs{ \langle L,\varphi*\tilde{\phi}\rangle} \leq \norm{L}_{W^{-1,p}} \norm{\varphi*\tilde{\phi}}_{W^{1,p'}}\leq \norm{L}_{W^{-1,p}} \norm{ \phi}_{W^{1,1}} \norm{\varphi}_{L^{p'}}.
$$
The thesis follows by taking the supremum over $\norm{\varphi}_{L^{p'}}\leq 1$.
\end{proof}

\subsection{Fractal dimensions}
Given any $A\subset \R^N$ and $\gamma \geq 0$, for any $\delta >0$ we set 
$$\mathcal{H}^\gamma_\delta (A) := \inf \left\{ \sum_i r_i^\gamma \, :\, A\subset \bigcup_i B_{r_i}, \, r_i<\delta \text{ for all } i \right\}.$$
Then, the $\gamma$-dimensional Hausdorff measure is defined as
\begin{equation}
    \mathcal{H}^\gamma(A) := \sup_{\delta > 0} \mathcal H^{\gamma}_\delta(A). 
\end{equation}
This is a non-negative Borel measure on $\R^N$ by the classical Carathéodory construction, and $\mathcal H^N$ is equivalent to the $N$-dimensional Lebesgue measure. The Hausdorff dimension is obtained as
\begin{equation}
    \dim_{\mathcal{H}}A: = \inf \{ \gamma\geq 0\,:\,\mathcal{H}^\gamma(A) = 0 \}. 
\end{equation}
Similarly, the $\gamma$-dimensional upper Minkowski content is defined as
$$\overline{\mathcal{M}}^{\gamma}(A) := \limsup_{r \to 0} \frac{\mathcal{H}^N([A]_r)}{r^{N-\gamma}} \qquad \text{with }[A]_r = \left\{ x \in \R^N \colon \dist(x,A) < r\right\}.$$
Then, the corresponding upper Minkowski dimension is given by
$$\overline{\dim }_{\mathcal{M}} A := \inf \{\gamma \geq 0\,:\,\overline{ \mathcal{M}}^\gamma(A) = 0\}.  $$

\section{Proof of the main results} \label{s: proofs}

We begin by proving the decomposition \eqref{D_decomp} in the more general case of solutions to the Navier--Stokes equations
\begin{equation}\label{NS} \tag{NS}
\left\{\begin{array}{l}
\partial_t u^\nu +\div (u^\nu \otimes u^\nu) +\nabla q^\nu=\nu \Delta u^\nu \\
\div u^\nu=0
\end{array}\right. \qquad \text{in }\T^d \times (0,T).
\end{equation}
The incompressible Euler equations \eqref{E} correspond to the case $\nu =0$. In the viscous setting, for $u^\nu\in L^2_t H^1_x \cap L^3_{x,t}$, the energy balance reads as
\begin{equation}
    \label{enbal_NS}
(\partial_t - \nu \Delta)\frac{\abs{u^\nu}^2}{2} +\div  \left(\left(\frac{\abs{u^\nu}^2}{2}+q^\nu\right)u^\nu\right) + \nu \abs{\nabla u^\nu}^2=-D^\nu  \qquad \text{in } \mathcal D'_{x,t}.
\end{equation}
For any $L^2$ initial datum, solutions are known to exist in the class $u^\nu\in L^\infty_t L^2_x \cap L^2_t H^1_x$. These are known as Leray--Hopf weak solutions \cites{L34,H51}. Weak solutions for which \eqref{enbal_NS}  holds with a non-negative distribution $D^\nu$ are called ``suitable weak solutions" in  Caffarelli--Kohn--Nirenberg  \cite{CKN} or ``dissipative" in  Duchon--Robert  \cite{DR00}.

\begin{proposition}
    \label{P:decomposition_NS}
    Let $\nu \geq 0$. Let $u^\nu \in L^3_{x,t}$ be a weak solution to \eqref{NS} and let $D^\nu$ be the associated Duchon--Robert distribution defined by \eqref{enbal_NS}. If $\nu>0$ we additionally require $u^\nu\in  L^2_t H^1_x$. Letting $u^\nu_\ell$ be the space mollification of $u^\nu$, we set  
     \begin{align}
     E^{\ell,\nu}&:=\frac{\abs{u^\nu-u^\nu_\ell}^2}{2}, \\
     Q^{\ell,\nu} &:= \left(\frac{\abs{u^\nu-u^\nu_\ell}^2}{2} + (q^\nu-q^\nu_\ell) \right) (u^\nu-u^\nu_\ell) \\
     R^{\ell,\nu}&:= u^\nu_\ell\otimes u^\nu_\ell - (u^\nu\otimes u^\nu)_\ell \\
     C^{\ell,\nu} &:= (u^\nu-u^\nu_\ell) \cdot \div R^{\ell,\nu} + (u^\nu-u^\nu_\ell)\otimes (u^\nu-u^\nu_\ell):\nabla u^\nu_\ell.
     \end{align}
    Denote $\mathcal E^\nu:= D^\nu + \nu\abs{\nabla u^\nu}^2$. For any $\ell>0$ we have the identity
     \begin{equation}
         \label{D_decomp_NS}
         -\mathcal E^\nu= (\partial_t+u^\nu_\ell \cdot \nabla - \nu\Delta) E^{\ell,\nu} +  \div Q^{\ell,\nu} + C^{\ell,\nu} + \nu \abs{\nabla u^\nu_\ell}^2 -2\nu  \nabla u^\nu:\nabla u^\nu_\ell   \qquad \text{in } \mathcal D'_{x,t}.
     \end{equation}
\end{proposition}
Note that \eqref{D_decomp_NS} can be equivalently written as 
\begin{equation}\label{D_decomp_NS_new}
   -D^\nu= (\partial_t+u^\nu_\ell \cdot \nabla - \nu\Delta) E^{\ell,\nu} +  \div Q^{\ell,\nu} + C^{\ell,\nu} + \nu \abs{\nabla (u^\nu_\ell-u^\nu)}^2.
\end{equation}

\begin{proof}
To lighten the notation, we will suppress the superscript $\nu$ on all the quantities. Since we are implicitly assuming the pressure to be the unique zero-average solution to $-\Delta q=\div \div (u\otimes u)$, the Calderón--Zygmund estimates imply $q\in L^\frac{3}{2}_{x,t}$. This is enough to justify all the computations below.

Let $\varphi\in C^\infty_{x,t}$ be any compactly supported function. We will repeatedly use that both $u$ and $u_\ell$ are divergence-free. By straightforward manipulations on \eqref{enbal_NS} we get
\begin{align}
    \left\langle  \mathcal E, \varphi \right\rangle   & = \int \frac{\abs{u}^2}{2} (\partial_t+ \nu \Delta)\varphi + \left( \frac{\abs{u}^2}{2} +q\right) u \cdot \nabla \varphi 
    \\ & = \int  \frac{\abs{u-u_\ell}^2}{2} (\partial_t + u_\ell\cdot \nabla + \nu\Delta) \varphi + \left( \frac{\abs{u-u_\ell}^2}{2} + (q-q_\ell)\right) (u-u_\ell) \cdot \nabla \varphi 
    \\ & \quad  + \underbrace{\int u \cdot  \partial_t (u_\ell \varphi ) -  \varphi u\cdot \partial_t u_\ell - \frac{\abs{u_\ell}^2}{2} (\partial_t + \nu \Delta)\varphi - q_\ell u_\ell \cdot \nabla \varphi}_{I}
    \\ & \quad + \underbrace{\int  u\cdot u_\ell u \cdot \nabla \varphi - \frac{\abs{u_\ell}^2}{2} u \cdot \nabla \varphi +q_\ell u \cdot \nabla \varphi +qu_\ell \cdot \nabla \varphi + \nu u\cdot u_\ell \Delta \varphi.}_{II}
\end{align}
Recall that $u_\ell$ satisfies 
\begin{equation}\label{NS_moll}
    \partial_t u_\ell + \div(u_\ell \otimes u_\ell) + \nabla q_\ell - \nu \Delta u_\ell= \div R^\ell , \qquad R^\ell:=u_\ell\otimes u_\ell -(u\otimes u)_\ell.
\end{equation}
Then, the energy balance of $u_\ell$ reads as
\begin{equation}\label{NS_enbal_moll}
   ( \partial_t  - \nu \Delta)\frac{\abs{u_\ell}^2}{2} + \div \left( \left( \frac{\abs{u_\ell}^2}{2} +q_\ell\right) u_\ell \right)  + \nu \abs{\nabla u_\ell}^2 = u_\ell \cdot \div R^\ell. 
\end{equation}
Hence, using $u_\ell \varphi$ as a test function in the weak formulation of \eqref{NS}, together with \eqref{NS_moll} and \eqref{NS_enbal_moll}, we obtain
\begin{align}
    I& = \int  - u \otimes u  : \nabla (u_\ell \varphi) -q\div (u_\ell \varphi) - \nu u \cdot \Delta (u_\ell \varphi)\\
    & \quad + \int \varphi u \cdot \Big( \div(u_\ell \otimes u_\ell) +   \nabla q_\ell - \nu \Delta u_\ell - \div R^\ell\Big)
    \\ & \quad  - \int \left( \div\left( \frac{\abs{u_\ell}^2}{2} u_\ell \right) - u_\ell \cdot \div R^\ell + \nu \abs{\nabla u_\ell}^2  \right) \varphi. 
\end{align}
We integrate by parts the terms with the Laplacian 
\begin{equation}
    - \int u \cdot \Delta (u_\ell \varphi) + \varphi u \cdot \Delta u_\ell =  \int \nabla u : u_\ell \otimes \nabla \varphi + 2 \varphi \nabla u: \nabla u_\ell + \nabla u_\ell : u \otimes \nabla \varphi.
\end{equation}
Therefore
\begin{align}
I & = \int  - u \otimes u  : \nabla u_\ell \varphi - u \otimes u : u_\ell \otimes \nabla \varphi -qu_\ell \cdot \nabla \varphi - q_\ell u \cdot \nabla \varphi + \varphi u \cdot \div(u_\ell \otimes u_\ell) \\
& \quad + \nu\int \nabla u :u_\ell \otimes \nabla \varphi + \nabla u_\ell :u\otimes \nabla \varphi
    \\ & \quad- \int \left( (u-u_\ell) \cdot \div R^\ell + \div\left( \frac{\abs{u_\ell}^2}{2} u_\ell \right)  +\nu \abs{\nabla u_\ell}^2 - 2\nu \nabla u:\nabla u_\ell\right) \varphi.
\end{align}
We aim to compute $I+II$. Note that
\begin{align}
\int u\cdot u_\ell \Delta \varphi &= - \int \left(\nabla u:u_\ell \otimes \nabla \varphi +  \nabla u_\ell:u \otimes \nabla \varphi \right), \\
- \int \frac{\abs{u_\ell}^2}{2} u \cdot \nabla \varphi & = \int \varphi u\cdot \nabla \frac{\abs{u_\ell}^2}{2}.
\end{align}
Thus, since $a\otimes b : c\otimes d=(a\cdot c) ( b\cdot d)$ and noticing that all the terms with $q$ and $q_\ell$ cancel out, we achieve
\begin{align}
     I+II & = \int \left(- u \otimes u : \nabla u_\ell + u \cdot \div(u_\ell \otimes u_\ell) \right) \varphi\\
     & \quad +\int  \left(   u\cdot  \nabla \frac{\abs{u_\ell}^2}{2}  - \div \left( \frac{\abs{u_\ell}^2}{2} u_\ell \right) +  (u_\ell - u) \cdot \div R^\ell \right) \varphi  \\
     &\quad -\nu \int \Big(  \abs{\nabla u_\ell}^2 - 2 \nabla u:\nabla u_\ell\Big) \varphi.
\end{align}
A direct computation shows that
\begin{align}
    - u \otimes u : \nabla u_\ell & + u \cdot \div(u_\ell \otimes u_\ell)  +u\cdot \nabla \frac{\abs{u_\ell}^2}{2} 
   - \div \left( \frac{\abs{u_\ell}^2}{2} u_\ell \right) = - (u - u_\ell) \otimes (u-u_\ell) : \nabla u_\ell,
\end{align}
concluding the proof. 
\end{proof}

We are now ready to prove Proposition \ref{P:quant moll est}, Theorem \ref{T:main D regularity} and Corollary \ref{C:main intermittency}. The symbol $*$ will always denote the space-time convolution. 

\begin{proof}[Proof of Proposition \ref{P:quant moll est}]
    Fix $\varphi\in C^\infty_{x,t}$ with compact support and find $\delta_0>0$ small enough so that $\langle D*\rho_\delta,\varphi\rangle$ is well-defined $\forall \delta<\delta_0$.
By the identity \eqref{D_decomp} we have
\begin{align}
    \abs{\left\langle   D-D*\rho_\delta,\varphi\right\rangle}   &=\abs{\left\langle   D, \varphi-\varphi*\rho_\delta\right\rangle  } \\
    &\leq \norm{E^\ell}_{L^\frac{p}{2}_{x,t}} \left( \norm{\partial_t (\varphi - \varphi * \rho_\delta)}_{L^{\frac{p}{p-2}}_{x,t}} + \norm{u_\ell}_{L^p_{x,t}} \norm{\nabla (\varphi - \varphi * \rho_\delta)}_{L^{\frac{p}{p-3}}_{x,t}}\right)\\
    &\quad +\norm{Q^\ell}_{L^\frac{p}{3}_{x,t}}\norm{\nabla (\varphi - \varphi * \rho_\delta)}_{L^{\frac{p}{p-3}}_{x,t}} + \norm{C^\ell}_{L^\frac{p}{3}_{x,t}}\norm{\varphi - \varphi * \rho_\delta}_{L^{\frac{p}{p-3}}_{x,t}}.
\end{align}

Thus, by using \eqref{moll_est_1}, \eqref{est_Eell}, \eqref{est_Qell} and \eqref{est_Cell} we deduce
\begin{align}
    \abs{\left\langle   D-D*\rho_\delta,\varphi\right\rangle }&\lesssim \left( \ell^{2\sigma} + \ell^{3\sigma}\right) \norm{\varphi}_{W^{1,\frac{p}{p-3}}_{x,t}}+ \ell^{3\sigma-1}\delta \norm{\varphi}_{W^{1,\frac{p}{p-3}}_{x,t}} \lesssim \delta^{\frac{2\sigma}{1-\sigma}} \norm{\varphi}_{W^{1,\frac{p}{p-3}}_{x,t}},
\end{align}
where we have chosen $\ell^{1-\sigma}=\delta$ to obtain the last estimate. Similarly 
\begin{align}
    \abs{\left\langle   D*\rho_\delta,\varphi\right\rangle } & = \abs{ \left\langle   D, \varphi*\rho_\delta\right\rangle } \\
    &\leq \norm{E^\ell}_{L^\frac{p}{2}_{x,t}} \left( \norm{\partial_t  \varphi * \rho_\delta}_{L^{\frac{p}{p-2}}_{x,t}} + \norm{u_\ell}_{L^p_{x,t}} \norm{\nabla \varphi * \rho_\delta}_{L^{\frac{p}{p-3}}_{x,t}}\right)\\
    &\quad +\norm{Q^\ell}_{L^\frac{p}{3}_{x,t}}\norm{\nabla \varphi * \rho_\delta}_{L^{\frac{p}{p-3}}_{x,t}} + \norm{C^\ell}_{L^\frac{p}{3}_{x,t}}\norm{ \varphi * \rho_\delta}_{L^{\frac{p}{p-3}}_{x,t}}\\
    &\lesssim \left( \ell^{2\sigma} + \ell^{3\sigma}\right)\delta^{-1}\norm{\varphi}_{L^{\frac{p}{p-3}}_{x,t}}+ \ell^{3\sigma-1} \norm{\varphi}_{L^{\frac{p}{p-3}}_{x,t}}\\
    &\lesssim \delta^{\frac{2\sigma}{1-\sigma}-1} \norm{\varphi}_{L^{\frac{p}{p-3}}_{x,t}},
\end{align}
where we have chosen again $\ell^{1-\sigma}=\delta$. 
\end{proof}

\begin{proof}[Proof of Theorem \ref{T:main D regularity}] We break the proof down into steps.

 \underline{\textsc{Step 0}}: Localization. 
 
Since we are working in a bounded time interval, we need to localize. Multiplying $D$ by a smooth compactly supported function, we can think of $D$ as a distribution on the whole space $\R^{d+1}$ with compact support. The presence of an additional smooth multiplicative function does not affect the validity of Proposition \ref{P:quant moll est}. In particular, the localization of $D$ satisfies the very same estimates on $\R^{d+1}$ without the restriction $\delta< \delta_0$. By duality, they read as
\begin{equation}
    \label{duality_norm_D_moll}
    \norm{ D-D* \rho_\delta}_{W^{-1,\frac{p}{3}}_{x,t}} \lesssim \delta^{\frac{2\sigma}{1-\sigma}} \qquad \text{and} \qquad  \norm{   D* \rho_\delta}_{L^\frac{p}{3}_{x,t}} \lesssim \delta^{\frac{2\sigma}{1-\sigma}-1}
\end{equation}
for all $\delta>0$.

\underline{\textsc{Step 1}}: Proof of $D\in B^{\frac{2\sigma}{1-\sigma}-1}_{\frac{p}{3},\infty}$.

Recall the definition of the Besov norm \eqref{besov_norm}. Since $\psi\in \mathcal S_{x,t}$, the estimate on $D*\psi$ is trivial by using \eqref{enbal_E} directly, i.e. the fact that $D$ is a space-time divergence of a $L^{\frac{p}{3}}_{x,t}$ vector field suffices. We are left to prove
\begin{equation}
    \label{goal_besov_seminorm}
    \sup_{k\geq 1}\left( 2^{\left(\frac{2\sigma}{1-\sigma} -1\right)k} \norm{ D* \phi_k}_{L^\frac{p}{3}_{x,t}}\right)<\infty.
\end{equation}
A direct consequence of \eqref{phi_k} is 
\begin{equation}
    \label{phi_k_L1_uniform}
    \sup_{k\geq 1} \left(\norm{\phi_k }_{L^1_{x,t}} + 2^{-k} \norm{ \nabla_{x,t} \phi_k }_{L^1_{x,t}} \right)<\infty.
\end{equation}
For any $k\geq 1$ and any $\delta>0$ we split
$$
\norm{ D* \phi_k}_{L^\frac{p}{3}_{x,t}}\leq \norm{ (D-D*\rho_\delta)* \phi_k}_{L^\frac{p}{3}_{x,t}} +  \norm{ (D*\rho_\delta)* \phi_k}_{L^\frac{p}{3}_{x,t}}.
$$
By the Young's inequality for convolutions and the second estimate in \eqref{duality_norm_D_moll} we deduce
$$
 \norm{ (D*\rho_\delta)* \phi_k}_{L^\frac{p}{3}_{x,t}}\leq  \norm{ D*\rho_\delta}_{L^\frac{p}{3}_{x,t}} \norm{ \phi_k}_{L^1_{x,t}}\lesssim \delta^{\frac{2\sigma}{1-\sigma} -1}.
$$
Similarly, by Lemma \ref{L:young} and the first estimate in \eqref{duality_norm_D_moll} we get
$$ \norm{ (D-D*\rho_\delta)* \phi_k}_{L^\frac{p}{3}_{x,t}} \leq \norm{ D-D*\rho_\delta}_{W^{-1,\frac{p}{3}}_{x,t}} \norm{  \phi_k}_{W^{1,1}_{x,t}} \lesssim \delta^{\frac{2\sigma}{1-\sigma} }2^k.
$$
These estimates together imply 
$$
\norm{D* \phi_k}_{L^\frac{p}{3}_{x,t}}\lesssim \delta^{\frac{2\sigma}{1-\sigma} -1} + \delta^{\frac{2\sigma}{1-\sigma} }2^k,
$$
from which \eqref{goal_besov_seminorm} follows by the choice $\delta=2^{-k}$.

\underline{\textsc{Step 2}}: Proof of $(i)$.

Let $\gamma\geq 0$ be such that \eqref{gamma_condition} holds and let $A$ be a Borel set such that $\mathcal H^\gamma (A)=0$. Since $D$ is a Radon measure, we can assume $A$ to be compact without loss of generality. By the Hahn decomposition we find two disjoint sets $S_+, S_-$ such that $D\llcorner S_{\pm}\geq 0$ and $D=D\llcorner S_{+} - D\llcorner S_{-}$. The goal is to show $D\llcorner S_{\pm} (A)=D(S_{\pm}\cap A)=0$. This implies $\abs{D}(A)= D\llcorner S_{+}(A) + D\llcorner S_{-}(A)=0$, concluding the proof. We will only prove $D(S_+\cap A)=0$ since the case of $S_-$ is analogous. 

Let $\eps>0$. Since $D$ is a Radon measure we find an open set $O\supset S_+\cap A$ such that 
\begin{equation}
    \label{D_small_complement}
    \abs{D}(O\setminus (S_+\cap A))<\eps.
\end{equation}
By $\mathcal H^\gamma(A)=0$ and $A$ compact, we find a finite family of balls $\left\{ B_{r_i} \right\}_i$ in $\R^{d+1}$ with $r_i < 1$ such that 
\begin{equation}
    \label{hausdorff_small}
    S_+\cap A \subset A\subset  \bigcup_i B_{r_i}\subset \bigcup_i B_{4r_i}\subset O, \qquad \sum_i r_i^\gamma <\eps.
\end{equation}
On each ball we define a cutoff $\chi^i\in C^\infty_{x,t}$ such that
\begin{equation}
    \label{cutoff_i}
    0\leq \chi^i\leq 1,\qquad \chi^i \big|_{B_{r_i}}\equiv 1,\qquad \chi^i \big|_{B^c_{2r_i}}\equiv 0 \qquad \text{and} \qquad \abs{\nabla_{x,t}\chi^i} \leq 4 r_i^{-1}.
\end{equation}
Set $\chi:=\max_i \chi^i$ pointwise. Clearly $\chi$ is Lipschitz continuous with compact support in $O$ and thus $D$ can be applied to it. We split
$$
D(S_+\cap A)=\int_{S_{+}\cap A} \chi \,dD=\int_O \chi \,dD- \int_{O\setminus (S_+\cap A)} \chi \,dD.
$$
Since $\abs{\chi}\leq 1$, by \eqref{D_small_complement} we get 
$$
\abs{\int_{O\setminus (S_+\cap A)} \chi \,dD } \leq \abs{D}(O\setminus (S_+\cap A))<\eps.
$$
We are left to estimate $\abs{\left\langle   D, \chi\right\rangle }$. 
Note that the second estimate in \eqref{moll_est_dissipation} implies $D\equiv 0$ whenever $\sigma>\frac13$. Thus, assume $\sigma\leq \frac13$. Furthermore, we can also restrict to $p>3$, since $p=3$ forces $\sigma >\frac13$ by \eqref{gamma_condition}. Recall that $B^{-\beta}_{b,\infty}=(B^\beta_{b',1})^*$ for all $\beta\in \R$, $b>1$ (see for instance \cite{BL76}*{Corollary 6.2.8}). By \textsc{Step 1} we infer that\footnote{We are using the embedding $W^{\beta+\alpha,b}\subset B^{\beta}_{b,1}$ for all $\beta\in \R$, $\alpha>0$ and $b\in [1,\infty]$. See for instance \cite{BL76}*{Theorem 6.2.4}.}
$$
\abs{\left\langle   D,\chi\right\rangle } \lesssim \norm{\chi}_{B^{1-\frac{2\sigma}{1-\sigma}}_{\frac{p}{p-3},1}} \lesssim \norm{\chi }_{W^{1-\frac{2\sigma}{1-\sigma}+\alpha,\frac{p}{p-3}}_{x,t}},
$$
where $\alpha>0$ is a small parameter to be fixed later. By the sub-additivity of the $\frac{p}{p-3}-$th power of the fractional Sobolev norm on the maximum between two functions (see for instance \cite{Wma15}*{Lemma 2.8}) we deduce\footnote{We are assuming $\alpha$ small enough so that $1- \frac{2\sigma}{1-\sigma} + \alpha < 1$. This is possible since $\sigma>0$.}
$$
\abs{\left\langle   D,\chi\right\rangle }  \lesssim \left(\sum_i  \norm{\chi^i}^{\frac{p}{p-3}}_{W^{1-\frac{2\sigma}{1-\sigma}+\alpha,\frac{p}{p-3}}_{x,t}} \right)^\frac{p-3}{p}.
$$
By interpolating the fractional Sobolev norm, the choice of $\chi^i$ made in \eqref{cutoff_i} implies
$$
\norm{\chi^i}_{W^{1-\frac{2\sigma}{1-\sigma}+\alpha,\frac{p}{p-3}}_{x,t}}\leq \norm{ \chi^i}^{\frac{2\sigma}{1-\sigma}-\alpha}_{L^{\frac{p}{p-3}}_{x,t}} \norm{\nabla_{x,t}\chi^i}^{1-\frac{2\sigma}{1-\sigma}+\alpha}_{L^{\frac{p}{p-3}}_{x,t}}\lesssim r_i^{(d+1)\frac{p-3}{p} -1+\frac{2\sigma}{1-\sigma}-\alpha}
$$
where the implicit constant does not depend on $i$. Thus
$$
\abs{\left\langle   D,\chi\right\rangle } \lesssim \left(\sum_i r_i^{ \left(\frac{2\sigma}{1-\sigma}-1-\alpha\right) \frac{p}{p-3} + d+1} \right)^\frac{p-3}{p}.
$$
Since $\gamma$ satisfies the strict inequality \eqref{gamma_condition}, we can choose $\alpha = \alpha (\sigma, \gamma, d, p)>0$ small enough so that
$$ 
\left(\frac{2\sigma}{1-\sigma}-1-\alpha\right) \frac{p}{p-3} + d+1>\gamma. 
$$
This, together with \eqref{hausdorff_small}, yields  
$$
\abs{\left\langle   D,\chi\right\rangle} \lesssim \left(\sum_i r_i^{ \gamma} \right)^\frac{p-3}{p}< \eps^\frac{p-3}{p}.
$$
The proof is concluded since $\eps>0$ is arbitrary and $p>3$.

\underline{\textsc{Step 3}}: Proof of $(ii)$. 

Let $K$ be a compact set and fix any point $(x,t)\in K$. Pick $\chi^r \in C^\infty_{x,t}$ such that
$$ 0\leq \chi^r\leq 1,\qquad \chi^r \big|_{B_r(x,t)}\equiv 1,\qquad \chi^r \big|_{B^c_{2r}(x,t)}\equiv 0 \qquad \text{and} \qquad \abs{\nabla_{x,t}\chi^r}\leq 4 r^{-1}. $$
This choice of $\chi^r$ implies
    $$ \norm{\chi^r}_{L^\frac{p}{p-3}_{x,t}} + r \norm{\nabla_{x,t}\chi^r}_{L^\frac{p}{p-3}_{x,t}}\lesssim r^{\frac{p-3}{p}(d+1)}.
    $$
    Since $D\geq 0$, then $\left\langle   D,\chi^r\right\rangle  $ gives an upper bound for $D(B_r(x,t))$. By \eqref{moll_est_dissipation}, for any $\delta>0$ we bound
    \begin{align}
    \left\langle   D,\chi^r\right\rangle   &\leq \abs{\left\langle   D- D*\rho_\delta,\chi^r\right\rangle} +\abs{\left\langle    D*\rho_\delta,\chi^r\right\rangle}\\
    &\lesssim \delta^{\frac{2\sigma}{1-\sigma}}\norm{\chi^r}_{W^{1,\frac{p}{p-3}}_{x,t}}+ \delta^{\frac{2\sigma}{1-\sigma}-1}\norm{\chi^r}_{L^{\frac{p}{p-3}}_{x,t}} \\
    &\lesssim \delta^{\frac{2\sigma}{1-\sigma}} r^{\frac{p-3}{p}(d+1)-1}+\delta^{\frac{2\sigma}{1-\sigma}-1}r^{\frac{p-3}{p}(d+1)}\\
    &\lesssim r^{\frac{2\sigma}{1-\sigma}-1+\frac{p-3}{p}(d+1)},
\end{align}
where the last inequality is obtained by the choice $\delta=r$.
\end{proof}

\begin{remark}
    \label{R:endpoint}
    As stated, Theorem \ref{T:main D regularity} directly yields  $D\equiv 0$ if $\sigma >\frac13$. Indeed, $\mathcal H^\gamma\equiv 0$ on $\R^{d+1}$ if $\gamma>d+1$, a choice allowed by \eqref{gamma_condition}. When $p>3$, at the endpoint $\sigma =\frac13$  the conclusion can be strengthen to $D\in L^\frac{p}{3}_{x,t}$. Indeed,  the identity  \eqref{D_decomp} implies that $D$ is the distributional limit of $C^\ell$, which is bounded in $L^\frac{p}{3}_{x,t}$ by \eqref{est_Cell}. In particular $D\ll \mathcal H^{d+1}$. This was first noted in \cite{Is24}.
\end{remark}

\begin{remark}\label{R:density bound}
   Denote by $\tilde \gamma$ the exponent of $r$ in \eqref{density bound}. Such uniform estimate yields a $L^\infty_{x,t}$ bound on the upper $\tilde \gamma$-dimensional density of $D$, which is a strictly stronger property than $D\ll \mathcal H^{\tilde \gamma}$ (see \cite{Magg12}*{Chapter 6}). Therefore, one should not expect to deduce an optimal absolute continuity from an estimate like \eqref{density bound}. Indeed, and as shown in $(i)$, $D\ll \mathcal H^\gamma$ for some $\gamma>\tilde \gamma$. The two numerology together suggest that densities of $D$ with respect to $\mathcal H^\gamma$ belong to $L^\frac{p}{3}_{x,t}$ when restricted to $\gamma$-dimensional sets.
\end{remark}

\begin{remark}
  The fact that $D\in B^{\frac{2\sigma }{1-\sigma}-1}_{\frac{p}{3},\infty}$ could have been proved directly by the definition of the Besov norm and the splitting \eqref{D_decomp}. Indeed, by \eqref{D_decomp} and the Young's inequality for  convolutions we have
  \begin{align}
  \norm{D*\phi_k}_{L^\frac{p}{3}_{x,t}}&\leq \norm{E^\ell}_{L^\frac{p}{2}_{x,t}}\left( \norm{\partial_t \phi_k}_{L^1_{x,t}} + \norm{u_\ell}_{L^p_{x,t}}\norm{\nabla\phi_k}_{L^1_{x,t}} \right) \\
  &\quad + \norm{Q^\ell}_{L^\frac{p}{3}_{x,t}}\norm{\nabla\phi_k}_{L^1_{x,t}} + \norm{C^\ell}_{L^\frac{p}{3}_{x,t}}\norm{\phi_k}_{L^1_{x,t}} \\
  &\lesssim \ell^{2\sigma} 2^k  + \ell^{3\sigma-1}.
  \end{align}
  Choosing $\ell^{1-\sigma}=2^{-k}$ yields  the desired bound.
  Then, the mollification estimates \eqref{moll_est_dissipation} become a consequence and  no longer the cause of $D\in B^{\frac{2\sigma }{1-\sigma}-1}_{\frac{p}{3},\infty}$. Of course, and not surprisingly, the two are equivalent.
\end{remark}

   \begin{remark}
       \label{R:negative measures sharp} In step $2$ in the proof of Theorem \ref{T:main D regularity} we have shown how a negative Besov regularity of a measure implies its absolute continuity with respect to a suitable $\mathcal H^\gamma$. This  cannot be improved in general. Indeed, measures with negative regularity saturating the corresponding intermittency constraint can be constructed (see for instance \cite{AH96}*{Chapters 2 \& 5}). The sharpness of the absolute continuity $|D|\ll \mathcal H^\gamma$ in the range \eqref{gamma_condition} for measures $D$ in the form \eqref{enbal_E} might be more subtle because of the space-time anisotropy.
   \end{remark}
   
\begin{remark}\label{R:distrib support}
    Assume that $D$ is only a distribution, thus not necessarily a measure. By the very same argument given in step $2$ in the proof of Theorem \ref{T:main D regularity} it can still be proved that, if $\dim_{\mathcal H} \left(\spt D\right)\leq \gamma$, then 
    $$
    \frac{2\sigma}{1-\sigma}>1- \frac{p-3}{p} (d+1-\gamma) \qquad \Longrightarrow \qquad D\equiv 0.
    $$
\end{remark}

We conclude with the proof of the intermittency corollary.

\begin{proof}[Proof of Corollary \ref{C:main intermittency}]
   For $\gamma=d+1$ the result is trivial by \cites{CET94,DR00} (see also Corollary \ref{C:cet} below). Therefore, let $\gamma\in [0,d+1)$. Assume that $D=D_{\rm ac} + D_{\rm sing}$ with $D_{\rm ac}\ll \mathcal{H}^{d+1}$ and $D_{\rm sing}$ concentrated on a set $S$ with $\dim_{\mathcal{H}} S= \gamma$. Being the two measures mutually singular, it holds $|D|=|D_{\rm ac}|+|D_{\rm sing}|$. We will prove the contrapositive. Assume there exist $p\in [3,\infty]$ and $\sigma_p\in (0,1)$ such that 
$$
\frac{2\sigma_p}{1-\sigma_p}>1-\frac
{p-3}{p} (d+1-\gamma) \qquad \text{ and } \qquad u\in L^p_t B^{\sigma_p}_{p,\infty}.
$$
Then, we can find $\gamma<\tilde \gamma<d+1$ such that the strict inequality stays true with $\gamma$ replaced by $\tilde \gamma$. Since $\dim_{\mathcal{H}} S= \gamma$ and $D_{\rm ac}\ll\mathcal{H}^{d+1}$, it must be $\mathcal{H}^{\tilde \gamma} (S)=|D_{\rm ac}|(S)=0$. Thus, $(i)$ from Theorem \ref{T:main D regularity} implies $\abs{D_{\rm sing}}(S)=|D|(S)=0$. On the other hand, the assumption $D_{\rm sing}\equiv D_{\rm sing}\llcorner S$ implies $|D_{\rm sing}|(S^c)=0$. Thus $\abs{D_{\rm sing}}\equiv 0$. 

\end{proof}


\section{Corollaries} \label{s: corollaries}
\subsection{The Euler equations}
Several corollaries directly follow from the splitting \eqref{D_decomp}.

\begin{corollary}[Euler energy conservation \cites{CET94}]\label{C:cet}
If $u\in L^3_t B^\sigma_{3,\infty}$ is a weak solution to \eqref{E} for some $\sigma>\frac{1}{3}$, then $D\equiv 0$.
\end{corollary}
\begin{proof}
    Since $u\in L^3_{x,t}$ and $q\in L^{\frac{3}{2}}_{x,t}$ we have
$$
\norm{E^\ell}_{L^1_{x,t}} + \norm{Q^\ell}_{L^1_{x,t}}\rightarrow 0 \qquad \text{as }\ell\rightarrow 0.
$$
Moreover, \eqref{est_Cell} implies $\norm{ C^\ell}_{L^1_{x,t}} \lesssim \ell^{3\sigma-1}$, which vanishes as $\ell\rightarrow 0$ since $\sigma>\frac{1}{3}$. This shows $D\equiv 0$.
\end{proof}

\begin{corollary}[Kinetic energy regularity \cite{Isett23}]\label{C:en reg}
Set
$$
e(t):=\frac{1}{2}\int_{\T^d} \abs{u(x,t)}^2 \,dx.
$$
If $u\in L^\infty_t B^\sigma_{3,\infty}\cap L^\infty_tB^\beta_{2,\infty}$ is a weak solution to \eqref{E} for some $\beta,\sigma\in(0,1)$, then\footnote{Since the spatial domain is bounded we have $\beta\geq \sigma$. Thus $1+2\beta-3\sigma>0$.} 
\begin{equation}\label{en_holder}
\abs{e(t)-e(s)}\lesssim \abs{t-s}^\frac{2\beta}{1+2\beta-3\sigma} \qquad \text{for a.e. } t,s.
\end{equation}
\end{corollary}
For $
\sigma<\frac13$, the H\"older regularity in \eqref{en_holder} is strictly larger when $\beta>\sigma$ with respect to the choice $\beta=\sigma$. Because of intermittency, the strict inequality $\beta>\sigma$ holds experimentally \cite{ISY20}. 
\begin{proof}
    By testing \eqref{D_decomp} with a compactly supported function $\eta\in C^\infty_t$ we deduce
    $$
    \int e \eta' \, dt =\left\langle   D ,\eta\right\rangle   = \int \int_{\T^d} \left(E^\ell \eta' - C^\ell \eta\right) \, dx\, dt.
    $$
    Thus, by  \eqref{est_Eell} and \eqref{est_Cell} we get 
    $$
    \abs{ \int e \eta'} \lesssim  \norm{\eta'}_{L^1_t} \ell^{2\beta} + \norm{\eta}_{L^1_t} \ell^{3\sigma-1},
    $$
    the implicit constant being independent of $\eta$.
    For a.e. $t,s$ (say $t>s$), by keeping $\|\eta'\|_{L^1_t}\lesssim 1$, we let $\eta\rightarrow \mathbbm{1}_{[s,t]}$ and obtain $\abs{e(t)-e(s)}\lesssim \ell^{2\beta} + \abs{t-s} \ell^{3\sigma-1}$. The choice $\ell^{1+ 2\beta-3\sigma}=\abs{t-s}$ concludes the proof.
\end{proof}

\begin{remark}
    If in Corollary \ref{C:en reg} we further assume\footnote{Here $L^2_{w,x}$ denotes the space of $L^2_x$ functions endowed with the weak topology.} $u\in C^0_t L^2_{w,x}$, then the thesis \eqref{en_holder} can be upgraded to hold for every $t,s$.
\end{remark}
\begin{corollary}[Minkowski Intermittency \cite{DRIS24}]\label{C:mink_int}
Let $p\in [3,\infty]$, $\sigma\in (0,1)$ and let $u\in L^p_tB^\sigma_{p,\infty}$ be a weak solution to \eqref{E} with Duchon--Robert distribution $D$ such that $\overline \dim_{\mathcal{M}}\left( \spt D\right)\leq \gamma$. Then 
$$
\frac{2\sigma}{1-\sigma}>1-\frac{p-3}{p}(d+1-\gamma) \qquad \Longrightarrow \qquad D\equiv 0.
$$
\end{corollary}
Note that, although the numerology coincides, Corollary \ref{C:mink_int} does not follow from $(i)$ in Theorem \ref{T:main D regularity} since we are not assuming $D$ to be a measure. 
\begin{proof}
    Let $\varphi\in C^\infty_{x,t}$ be any compactly supported test function. Denote by $[\spt D]_\eps$ the space-time $\eps$-neighbourhood of $\spt D$. Let $\chi^\eps\in C^\infty_{x,t}$ be such that 
    $$
    0\leq \chi^\eps\leq 1,\qquad \chi^\eps \big|_{[\spt D]_\eps}\equiv 1,\qquad \chi^\eps \big|_{[\spt D]_{2\eps}^c}\equiv 0 \qquad \text{and} \qquad \abs{\nabla_{x,t}\chi^\eps} \leq 4\eps^{-1}.
    $$
    Then $\left\langle   D,\varphi\right\rangle  =\left\langle   D,\varphi \chi^\eps\right\rangle  $, and  by the  decomposition \eqref{D_decomp} we infer
    \begin{align}
        \abs{\left\langle   D,\varphi\right\rangle}&\lesssim \left( \norm{E^\ell}_{L^{\frac{p}{2}}_{x,t}} \left(1+ \norm{u_\ell}_{L^p_{x,t}} \right) + \norm{Q^\ell}_{L^{\frac{p}{3}}_{x,t}} + \norm{C^\ell}_{L^{\frac{p}{3}}_{x,t}} \right) \norm{\chi^\eps}_{L^{\frac{p}{p-3}}_{x,t}}\\
        &+\left( \norm{E^\ell}_{L^{\frac{p}{2}}_{x,t}} \left(1+ \norm{u_\ell}_{L^p_{x,t}} \right)+ \norm{Q^\ell}_{L^{\frac{p}{3}}_{x,t}}  \right) \norm{\nabla_{x,t}\chi^\eps}_{L^{\frac{p}{p-3}}_{x,t}},
    \end{align}
    where the implicit constant depends only on $\varphi$. Let $\alpha>0$ be a small parameter that will be fixed later. The assumption $\overline \dim_{\mathcal{M}}\left( \spt D\right)\leq \gamma$ implies 
    $$
   \norm{\chi^\eps}_{L^{\frac{p}{p-3}}_{x,t}}+\eps\norm{\nabla_{x,t}\chi^\eps}_{L^{\frac{p}{p-3}}_{x,t}}\lesssim \eps^{\frac{p-3}{p}(d+1-\gamma- \alpha)},
    $$
    if $\eps$ is sufficiently small. Thus, by \eqref{est_Eell}, \eqref{est_Qell} and \eqref{est_Cell} we deduce 
    \begin{align}
         \abs{\left\langle   D,\varphi\right\rangle} &\lesssim \left( \ell^{2\sigma} + \ell^{3\sigma} + \ell^{3\sigma-1}\right) \eps^{\frac{p-3}{p}(d+1-\gamma -\alpha)} + \left( \ell^{2\sigma} + \ell^{3\sigma} \right) \eps^{\frac{p-3}{p}(d+1-\gamma-\alpha)-1}\\
         &\lesssim \ell^{3\sigma-1}\eps^{\frac{p-3}{p}(d+1-\gamma-\alpha)} + \ell^{2\sigma}\eps^{\frac{p-3}{p}(d+1-\gamma-\alpha)-1},
    \end{align}
    where the implicit constant depends only on norms of $\varphi$ and $u$. The choice $\ell^{1-\sigma}=\eps$ leads to
     \begin{align}
         \abs{\left\langle   D,\varphi\right\rangle  } & \lesssim \eps^{\frac{2\sigma}{1-\sigma}-1+\frac{p-3}{p}(d+1-\gamma-\alpha)}.
    \end{align}
    Since $\sigma,p,\gamma$ satisfy an open condition, we find $\alpha>0$ sufficiently small so that the exponent of $\eps$ in the above inequality is positive. The proof is concluded by letting $\eps\rightarrow 0$.
\end{proof}

\subsection{The Navier--Stokes equations}
We start with the following.
\begin{corollary}[Onsager quasi-singularity \cite{DE19}] 
Let $\{u^\nu\}_{\nu>0}\subset L^2_t H^1_x$ be a sequence of weak solutions to  \eqref{NS}. Denote by $\mathcal E^\nu:=D^\nu+\nu \abs{\nabla u^\nu}^2$ and pick any $\varphi\in C^\infty_{x,t}$ with compact support. Then
$$
\sup_{\nu>0} \norm{u^\nu}_{L^3_t B^\sigma_{3,\infty}}<\infty, \, \sigma \in \left[ \frac13,1 \right) \qquad \Longrightarrow \qquad \abs{ \left\langle   \mathcal E^\nu ,\varphi\right\rangle  } \lesssim \nu^\frac{3\sigma-1}{1+\sigma}.
$$
In particular, if 
$$
\liminf_{\nu\rightarrow 0} \frac{\abs{ \left\langle   \mathcal E^\nu ,\varphi\right\rangle  }}{\nu^\alpha} >0 \qquad \text{for some } \alpha\in [0,1),
$$
then 
$$
\liminf_{\nu\rightarrow 0} \norm{u^\nu}_{L^3_tB^{\sigma}_{3,\infty}}=\infty\qquad \forall \sigma>\frac{1+\alpha}{3-\alpha}.
$$
\end{corollary}
\begin{proof}
    We write the last two terms in the decomposition \eqref{D_decomp_NS} as
    \begin{align}
        \int \left(\abs{\nabla u^\nu_\ell}^2 - 2 \nabla u^\nu:\nabla u^\nu_\ell \right)\varphi &=\int \left(\abs{\nabla u^\nu_\ell}^2 - 2 \nabla (u^\nu-u^\nu_\ell):\nabla u^\nu_\ell -2 \abs{\nabla u^\nu_\ell}^2\right)\varphi\\
        &=-\int \abs{\nabla u^\nu_\ell}^2 \varphi + 2 \int (u^\nu - u^\nu_\ell)\cdot \Delta u^\nu_\ell \varphi \\
        &\quad + 2 \int (u^\nu-u^\nu_\ell)\otimes \nabla \varphi :\nabla u^\nu_\ell.\label{reass_last_term}
    \end{align}
    Since $\{u^\nu\}_{\nu>0}$ is bounded in $L^3_t B^\sigma_{3,\infty}$, by \eqref{est_Eell}, \eqref{est_Qell}, \eqref{est_Cell} and \eqref{moll_est_2} we estimate
    \begin{align}
        \abs{ \left\langle   \mathcal E^\nu ,\varphi\right\rangle }& \lesssim \norm{E^{\ell,\nu}}_{L^{\frac{3}{2}}_{x,t}} \left(1+ \norm{ u^\nu_\ell}_{L^3_{x,t}} \right) + \norm{Q^{\ell,\nu}}_{L^{1}_{x,t}} \\
        &\quad + \norm{C^{\ell,\nu}}_{L^{1}_{x,t}} + \nu \norm{E^{\ell,\nu}}_{L^{\frac{3}{2}}_{x,t}} + \nu \norm{\nabla u^\nu_\ell}^2_{L^2_{x,t}} \\
        & \quad +\nu  \norm{u^\nu-u^\nu_\ell}_{L^2_{x,t}} \left(\norm{\Delta u^\nu_\ell}_{L^2_{x,t}}+ \norm{\nabla u^\nu_\ell}_{L^2_{x,t}}\right)\\
        &\lesssim \ell^{2\sigma} + \ell^{3\sigma}+ \ell^{3\sigma-1}+\nu \ell^{2\sigma} + \nu \ell^{2(\sigma-1)} +\nu \ell^{2\sigma -1 }\\
        &\lesssim \ell^{3\sigma-1}+ \nu \ell^{2(\sigma-1)}.
    \end{align}
    The proof is concluded by choosing $\ell^{1+\sigma}=\nu$.
\end{proof}

The next result provides a precise physical meaning to the term $C^{\ell,\nu}$ in the decomposition \eqref{D_decomp_NS}

\begin{corollary}[Inertial flux equals dissipation]\label{C:ff law}
    Let $\{u^\nu\}_{\nu>0}\subset L^2_t H^1_x\cap L^3_{x,t}$ be a sequence of weak solutions to \eqref{NS}. Assume 
\begin{equation}\label{unif_besov_bound_2_ff_law}
\sup_{\nu>0} \norm{ u^\nu }_{L^2_t B^\sigma_{2,\infty}}<\infty \qquad \text{for some } \sigma \in (0,1).
\end{equation}
Let $\{\ell_\nu\}_{\nu>0}$ be any infinitesimal\footnote{Meaning that $\ell_\nu\rightarrow 0$ as $\nu\rightarrow 0$.} sequence of positive numbers such that 
\begin{equation}
    \label{k41_dissip_scale}
    \limsup_{\nu\rightarrow 0} \frac{\nu}{\ell_\nu^{2(1-\sigma)}}=0.
\end{equation}
Denote $\mathcal E^\nu:=D^\nu+\nu\abs{\nabla u^\nu}^2$. Then, for any  compactly supported $\eta\in C^\infty_t$, it holds 
\begin{equation}\label{ff_law_Cell}
\lim_{\ell_I\rightarrow 0} \limsup_{\nu\rightarrow 0} \sup_{\ell\in [\ell_\nu,\ell_I]} \abs{\left\langle   \mathcal E^\nu + C^{\ell, \nu},\eta\right\rangle   } =0.
\end{equation}
\end{corollary}

\begin{proof}
   Testing formula \eqref{D_decomp_NS} with a smooth time dependent function $\eta$ with compact support and writing the last term in \eqref{D_decomp_NS} as in \eqref{reass_last_term}, we estimate 
   \begin{align}
       \abs{\left\langle   \mathcal E^\nu+ C^{\ell, \nu},\eta\right\rangle }  \lesssim \norm{E^{\ell,\nu}}_{L^1_{x,t}} + \nu \left(\norm{\nabla u^\nu_\ell}^2_{L^2_{x,t}} + \norm{u^\nu-u^\nu_\ell}_{L^2_{x,t}}\norm{\Delta u^\nu_\ell}_{L^2_{x,t}}\right).
   \end{align}   
   Thus, \eqref{moll_est_1}, \eqref{moll_est_2} and \eqref{est_Eell} yield  
   $$
    \abs{\left\langle   \mathcal E^\nu + C^{\ell, \nu},\eta\right\rangle } \lesssim \ell^{2\sigma} + \nu \ell^{2(\sigma-1)}.
   $$
  Therefore
   $$
    \sup_{\ell\in [\ell_\nu,\ell_I]} \abs{ \left\langle   \mathcal E^\nu + C^{\ell, \nu},\eta\right\rangle } \lesssim \ell_I^{2\sigma} + \frac{\nu}{\ell_\nu^{2(1-\sigma)}},
   $$
   and our choice \eqref{k41_dissip_scale} of the dissipative length scale implies
   $$
  \limsup_{\nu\rightarrow 0} \sup_{\ell\in [\ell_\nu,\ell_I]}\abs{\left\langle  \mathcal E^\nu + C^{\ell, \nu},\eta\right\rangle } \lesssim \ell_I^{2\sigma}.
   $$
   The proof is concluded by letting $\ell_I\rightarrow 0$ since $\sigma>0$.
\end{proof}
\begin{remark}\label{R:ff_law_compactness}
    Corollary \ref{C:ff law} still holds by weakening the uniform Besov bound to just $L^2_{x,t}$ compactness of the sequence $\{u^\nu\}_{\nu>0}$. In this case, \eqref{k41_dissip_scale} has to be modified accordingly.
\end{remark}

\begin{remark}
    From \eqref{ff_law_Cell} we deduce that $C^{\ell,\nu}$ is the term that might cause anomalous dissipation in the inviscid limit. Indeed, if non-vanishing, it keeps the rate of dissipation of order $1$ all over the scales in the inertial range   $[\ell_\nu,\ell_I]$. Thus, in view of the Kolmogorov ``Four-Fifths Law", it must be equivalent to third order longitudinal increments, at least asymptotically over the relevant range of scales. To see this, set $\delta_{\ell z} u^\nu(x,t):=u^\nu (x+\ell z,t)-u^\nu(x,t)$ and define the local longitudinal third order structure function by
    $$
    S^\nu_{3,\parallel}(\ell)(x,t):=\fint_{\mathbb{S}^{d-1}} \Big( z\cdot \delta_{\ell z} u^\nu (x,t)\Big)^3\,d\mathcal H^{d-1}(z).
    $$
    Assuming \eqref{unif_besov_bound_2_ff_law} and \eqref{k41_dissip_scale}, the arguments from \cites{eyink2002local,Novack24} prove 
    $$
    \lim_{\ell_I\rightarrow 0} \limsup_{\nu\rightarrow 0} \sup_{\ell\in [\ell_\nu,\ell_I]} \abs{\left\langle   \frac{S^\nu_{3,\parallel}(\ell)}{\ell}+  \frac{12}{d(d+2)}\mathcal E^{ \nu},\eta\right\rangle} =0.
    $$
    Then, Corollary \ref{C:ff law} implies 
    $$
 \lim_{\ell_I\rightarrow 0} \limsup_{\nu\rightarrow 0} \sup_{\ell\in [\ell_\nu,\ell_I]} \abs{\left\langle   \frac{S^\nu_{3,\parallel}(\ell)}{\ell}- \frac{12}{d(d+2)}C^{\ell,\nu},\eta\right\rangle }=0,
$$
giving a precise structural meaning to $C^{\ell,\nu}$.  See \cite{drivas2022self} for possible implications for self-regularization of turbulence.
\end{remark}
\begin{remark}
    If in Corollary \ref{C:ff law} we additionally require  $\{u^\nu\}_{\nu>0}$  to be strongly compact in $L^3_{x,t}$ 
then the thesis holds locally in space-time, that is, one can choose any compactly supported $\varphi\in C^\infty_{x,t}$ in \eqref{ff_law_Cell} instead of just $\eta\in C^\infty_t$. 
\end{remark}

The next corollary shows that to rule out anomalous dissipation it is sufficient to retain an Onsager subcritical regularity in the dissipative range only.  For any $\ell>0$ we define 
$$
\overline S^\nu_3(\ell):=\int_0^T \fint_{B_\ell(0)} \|u^\nu(\cdot +z,t) - u^\nu(\cdot, t)\|^3_{L^3_x}\,dzdt.
$$

\begin{corollary}[A dissipative range condition]\label{C:dissipative range onsager}
     Let $\{u^\nu\}_{\nu>0}\subset L^2_t H^1_x\cap L^3_{x,t}$ be a sequence of weak solutions to \eqref{NS} such that $\sup_{\nu>0}\nu \int |\nabla u^\nu|^2<\infty$. Assume there exists an infinitesimal sequence of positive numbers $\{\ell_\nu\}_{\nu>0}$ and a value $\sigma\geq \frac13$ such that 
     \begin{equation}
         \label{condition on ell dissipative}
         \limsup_{\nu\rightarrow 0} \frac{\nu}{\ell_\nu^{2(1-\sigma)}}<\infty\qquad \text{and}\qquad \lim_{\nu\rightarrow 0}\frac{\overline S^\nu_3(\ell_\nu)}{\ell_\nu^{3\sigma}}=0.
     \end{equation}
    Denote $\mathcal{E}^\nu:=D^\nu+\nu|\nabla u^\nu|^2$.  For any $\eta\in C^\infty_t$  with compact support it holds
     $$
     \limsup_{\nu\rightarrow 0}\left| \langle\mathcal{E}^\nu,\eta \rangle\right|=0.
     $$
\end{corollary}
\begin{proof}
Since $\eta$ does not depend on the space variable, by \eqref{D_decomp_NS} we get 
    \begin{equation}
        \label{energy balance for time dependent test}
        \langle\mathcal{E}^\nu,\eta \rangle=\int E^{\ell,\nu}\partial_t\eta - \int C^{\ell,\nu}\eta - \nu \int |\nabla u^\nu_\ell|^2\eta +2\nu \int \nabla u^\nu :\nabla u^\nu_\ell\eta.
    \end{equation}
    Direct computations show
    \begin{align}
\|E^{\ell,\nu}\|_{L^\frac32_{x,t}}\lesssim\left(\overline S^\nu_3(\ell)\right)^\frac23,\qquad   \|C^{\ell,\nu}\|_{L^1_{x,t}}\lesssim   \frac{\overline S^\nu_3(\ell)}{\ell}\qquad \text{and}\qquad \|\nabla u^\nu_\ell\|_{L^3_{x,t}} \lesssim \frac{\left(\overline S^\nu_3(\ell)\right)^\frac13}{\ell}.\\
    \end{align}
    Thus, we can bound all the terms in \eqref{energy balance for time dependent test} to get 
    \begin{align}
        \left| \langle\mathcal{E}^\nu,\eta \rangle\right|&\lesssim \| E^{\ell,\nu}\|_{L^1_{x,t}}+ \| C^{\ell,\nu}\|_{L^1_{x,t}} +\nu \| \nabla u^\nu_\ell \|^2_{L^2_{x,t}} +\nu \| \nabla u^\nu \|_{L^2_{x,t}} \| \nabla u^\nu_\ell \|_{L^2_{x,t}}\\
        &\lesssim \| E^{\ell,\nu}\|_{L^\frac32_{x,t}}+ \| C^{\ell,\nu}\|_{L^1_{x,t}} +\nu \| \nabla u^\nu_\ell \|^2_{L^3_{x,t}} +\sqrt \nu  \| \nabla u^\nu_\ell \|_{L^3_{x,t}}\\
        &\lesssim \left(\overline S^\nu_3(\ell)\right)^\frac23 +\frac{\overline S^\nu_3(\ell)}{\ell} +  \nu \frac{\left(\overline S^\nu_3(\ell)\right)^\frac23}{\ell^2}+ \sqrt \nu \frac{\left(\overline S^\nu_3(\ell)\right)^\frac13}{\ell} \qquad \forall \ell>0.
    \end{align}
    Let $\{\ell_\nu\}_{\nu>0}$ be the infinitesimal sequence satisfying \eqref{condition on ell dissipative}. Since $\sigma\geq \frac13$, we get
    $$
   \lim_{\nu\rightarrow 0} \frac{\overline S^\nu_3(\ell_\nu)}{\ell_\nu}=0.
    $$
    Thus, by choosing $\ell=\ell_\nu$ and letting $\nu\rightarrow 0$ we conclude 
    $$
    \limsup_{\nu\rightarrow 0} \left| \langle\mathcal{E}^\nu,\eta\rangle\right|\lesssim  \limsup_{\nu\rightarrow 0}  \left(\frac{\nu}{\ell_\nu^{2(1-\sigma)}} \right)^\frac12\left(\frac{\overline S^\nu_3(\ell_\nu)}{\ell_\nu^{3\sigma}}\right)^\frac13=0.
    $$
\end{proof}

\begin{remark}
    Instead of $\eta\in C^\infty_t$, in Corollary \ref{C:dissipative range onsager} we could have chosen a test function $\varphi\in C^\infty_{x,t}$ depending on both space and time,  requiring in addition that $\{u^\nu\}_{\nu>0}\subset L^3_{x,t}$ is bounded. This is needed in order to estimate the advective term as 
    $$
    \left\|u^\nu_\ell E^{\ell,\nu}\right\|_{L^1_{x,t}}\leq \left\|u^\nu_\ell \right\|_{L^3_{x,t}} \left\|E^{\ell,\nu}\right\|_{L^\frac{3}{2}_{x,t}}\lesssim \left(\overline S^\nu_3(\ell)\right)^\frac23.
    $$
\end{remark}

The following corollary is also new, quantifying the relevant scales to capture the whole dissipation.
\begin{corollary}[Resolved dissipation scales]\label{C:resolved_diss}
    Let $\{u^\nu\}_{\nu>0}\subset L^2_t H^1_x$ be a sequence of weak solutions to \eqref{NS} such that $\sup_{\nu>0} \nu\int \abs{\nabla u^\nu}^2<\infty$. Assume 
$$
\sup_{\nu>0} \norm{u^\nu }_{L^4_t B^\sigma_{4,\infty}}<\infty \qquad \text{for some } \sigma\in  (0,1).
$$ 
Let $\{\ell_\nu\}_{\nu>0}$ be any  sequence of positive numbers such that 
\begin{equation}
    \label{k41_dissip_scale_new}
    \limsup_{\nu\rightarrow 0} \frac{\ell_\nu^{4\sigma}}{\nu}=0.
\end{equation}
Let $\eta\in C^\infty_t$ be  non-negative and compactly supported. Denote $\mathcal E^\nu:=D^\nu+\nu \abs{\nabla u^\nu}^2$. It holds 
$$
\limsup_{\nu\rightarrow 0} \nu \int \abs{\nabla u^\nu_{\ell_\nu} }^2 \eta =0 \qquad \Longrightarrow \qquad \limsup_{\nu\rightarrow 0}\, \abs{\left\langle  \mathcal E^\nu ,\eta\right\rangle} =0.
$$
\end{corollary}

\begin{proof}
By applying \eqref{D_decomp_NS} to the space independent function $\eta$ we get 
$$
\abs{\left\langle  \mathcal E^\nu ,\eta\right\rangle } \lesssim \norm{E^{\ell,\nu} \eta' }_{L^1_{x,t}} +  \norm{C^{\ell,\nu}\eta}_{L^1_{x,t}} + \nu \int \abs{\nabla u^\nu_\ell}^2 \eta +\left( \nu \int \abs{\nabla u^\nu_\ell}^2 \eta\right)^\frac12,
$$
where we have used that $\sup_{\nu>0} \nu\int \abs{\nabla u^\nu}^2<\infty$ by assumption.  Choose $\ell=\ell_\nu$. Since $\{u^\nu\}_{\nu>0}$ is bounded in $L^2_t B^\sigma_{2,\infty}$, we use \eqref{est_Eell} to estimate the first term. Thus, we infer
\begin{align}
    \limsup_{\nu \to 0} \abs{\left\langle  \mathcal E^\nu ,\eta\right\rangle } & \lesssim \limsup_{\nu \to 0 }\left(\ell_\nu^{2\sigma} + \norm{C^{\ell_\nu,\nu}\eta}_{L^1_{x,t}} + \nu \int \abs{\nabla u^\nu_{\ell_\nu}}^2 \eta + \left( \nu \int \abs{\nabla u^\nu_{\ell_\nu}}^2 \eta\right)^\frac12 \right)
    \\ & \leq \limsup_{\nu \to 0} \norm{C^{\ell_\nu, \nu} \eta}_{L^{1}_{x,t}}. 
\end{align}
Note that so far we have only used that $\{u^\nu\}_{\nu>0}$ is bounded in $L^2_tB^\sigma_{2,\infty}$. The stronger assumption $L^4_tB^\sigma_{4,\infty}$ is needed to handle $\norm{C^{\ell_\nu,\nu}\eta}_{L^1_{x,t}}$. Indeed, we estimate
\begin{align}
\norm{C^{\ell_\nu,\nu}\eta}_{L^1_{x,t}}&\lesssim \norm{u^\nu-u^\nu_{\ell_\nu} }_{L^4_{x,t}}^2 \norm{\nabla u^\nu_{\ell_\nu} \sqrt{\eta}}_{L^2_{x,t}} + \norm{u^\nu-u^\nu_{\ell_\nu}}_{L^4_{x,t}} \norm{\eta^\frac34 \div R^{\ell_\nu,\nu}}_{L^\frac{4}{3}_{x,t}} \\
&\lesssim \ell_\nu^{2\sigma} \norm{\nabla u^\nu_{\ell_\nu} \sqrt{\eta}}_{L^2_{x,t}} + \ell_\nu^{2\sigma} \left(  \int \abs{\nabla u^\nu}^2 \eta\right)^\frac12.
\end{align}
The last inequality follows by the commutator estimate
$$
\norm{\div R^{\ell_\nu,\nu}}_{L^\frac{4}{3}_{x}}\lesssim \ell_\nu^{\sigma} \norm{u^\nu}_{B^\sigma_{4,\infty}} \norm{\nabla u^\nu }_{L^2_x},
$$
which can be obtained by writing $\div R^{\ell_\nu,\nu} = u^\nu_{\ell_\nu}\cdot \nabla u^\nu_{\ell_\nu} - (u^\nu\cdot \nabla u^\nu)_{\ell_\nu}$. We conclude
\begin{align}
\limsup_{\nu\rightarrow 0}\abs{\left\langle  \mathcal E^\nu ,\eta\right\rangle } \lesssim \limsup_{\nu\rightarrow 0} \left(\left(\frac{\ell_\nu^{4\sigma}}{\nu} \nu \int \abs{\nabla u^\nu_{\ell_\nu}}^2 \eta \right)^{\frac12} +\left(\frac{\ell_\nu^{4\sigma}}{\nu} \nu \int \abs{\nabla u^\nu}^2 \eta \right)^{\frac12}\right)=0.
\end{align} 
\end{proof}

\begin{remark}
    In Corollary \ref{C:resolved_diss}, instead of $\eta\in C^\infty_t$, we could have chosen a space-time test function $\varphi\in C^\infty_{x,t}$. This only causes the appearance of three more local terms which do not require any additional assumption to be estimated.
\end{remark}
\begin{remark}
    When $\sigma=\frac13$ both the scales from Corollary \ref{C:ff law} and Corollary \ref{C:resolved_diss} can  almost be chosen as $\ell_\nu\sim \nu^{\frac34}$, i.e. the Kolmogorov dissipative length scale. In Corollary \ref{C:dissipative range onsager} one can choose  $\ell_\nu=\nu^\frac34$ exactly. Any scale asymptotically smaller than $\nu^\frac34$ would not be compatible with the first condition in \eqref{condition on ell dissipative} and, in order to rule out anomalous dissipation, the Onsager subcritical regularity must be  asked (at least) in the whole dissipative range.
\end{remark}

\begin{corollary}[Shinbrot local energy balance \cite{S74}]
 Let $u^\nu\in L^\infty_t L^2_x\cap L^2_t H^1_x \cap L^3_{x,t}$ be a weak solution to \eqref{NS}. Assume $u^\nu\in L^m_tL^p_x$ for some $m,p$ such that $\frac{2}{p}+\frac{2}{m}\leq 1$, $p\geq 4$. Then \eqref{enbal_NS} holds with $D^\nu\equiv 0$.
\end{corollary}
\begin{proof}
Since viscosity is fixed, we omit the superscripts $\nu$. Since $u\in L^3_{x,t}\cap L^2_tH^1_x$ and $q\in L^{\frac32}_{x,t}$, we have
$$
\norm{E^\ell}_{L^\frac{3}{2}_{x,t}} + \norm{Q^\ell}_{L^1_{x,t}} + \norm{\nabla (u_\ell - u)}_{L^2_{x,t}}\rightarrow 0 \qquad \text{as } \ell\rightarrow 0.
$$
Thus, by \eqref{D_decomp_NS_new}, $D\equiv 0$ whenever $ \norm{C^\ell}_{L^1_{x,t}} \to  0$ as $\ell \to 0$. Recall that 
 $$
 C^\ell=(u-u_\ell)\cdot \div R^\ell + (u-u_\ell)\otimes (u-u_\ell):\nabla u_\ell=: C^\ell_1+ C^\ell_2.
 $$ 
By Lemma \ref{L:shin} we  estimate
\begin{equation}
 \norm{C^\ell_2}_{L^1_{x,t}}\leq \norm{u-u_\ell}_{L^m_tL^p_x} \norm{\nabla u_\ell}_{L^2_{x,t}} \norm{u-u_\ell}_{L^\infty_tL^2_x}^{2-\frac{m}{2}} \norm{u-u_\ell}_{L^m_t L^p_x}^{\frac{m}{2}-1}.\label{ugly_holder}
 \end{equation}
 Since we are working locally, by possibly reducing $m$ we can assume  $\frac{2}{p}+\frac{2}{m}=1$. In particular, since $p\geq 4$, it must be $m<\infty$. It is then clear that the expression in the right hand side of \eqref{ugly_holder} goes to zero as $\ell\rightarrow 0$, if $p<\infty$, by standard properties of mollifiers. If $p=\infty$, then $m=2$ necessarily. Thus $u\in L^\infty_tL^2_x \cap L^2_tL^\infty_x\subset L^4_{x,t}$ and we can conclude as before.  Moreover, since $u\in L^2_tH^1_x$, the term $C^\ell_1$ can be written as 
 $$
 C^\ell_1=(u-u_\ell) \otimes u_\ell :\nabla u_\ell - (u-u_\ell)\cdot (u\cdot \nabla u)_\ell.
 $$
 Thus, the very same argument we have given for $C^\ell_2$ applies to show $\lim_{\ell\rightarrow 0} \norm{C^\ell_1}_{L^1_{x,t}}=0$.
\end{proof}

The next and last corollary is also new. It provides a uniform bound of the viscous dissipation in a negative Besov space.

\begin{corollary}[Uniform bound viscous dissipation]\label{C:unif_visc_bound}
Let $\{u^\nu\}_{\nu>0}\subset L^2_tH^1_x$ be a sequence of weak solutions to \eqref{NS}. Denote by $\mathcal{E}^\nu:=D^\nu + \nu \abs{\nabla u^\nu}^2$. If $\{u^\nu\}_{\nu>0}$ is bounded in $L^p_tB^\sigma_{p,\infty}$  for some $p\in [3,\infty]$ and $\sigma \in (0,1)$, then for any $\varphi\in C^\infty_{x,t}$ compactly supported $\exists\delta_0>0$, which depends only on the distance of the support of $\varphi$ from the boundary of the space-time domain, such that 
\begin{equation}\label{moll_est_dissipation_NS}
            \abs{\left\langle   \mathcal{E}^\nu-\mathcal{E}^\nu* \rho_\delta,\varphi\right\rangle} \lesssim \delta^{2\sigma} \norm{\varphi }_{W^{2,\frac{p}{p-3}}_{x,t}} \qquad \text{and} \qquad \abs{\left\langle   \mathcal{E}^\nu* \rho_\delta,\varphi\right\rangle } \lesssim \delta^{2\sigma-2} \norm{\varphi}_{L^{\frac{p}{p-3}}_{x,t}}
        \end{equation}
        for all $\delta<\delta_0$, with implicit constants that are uniform in viscosity. In particular, $\left\{ \mathcal{E}^\nu\right\}_{\nu>0}$ is bounded in $B^{2(\sigma-1)}_{\frac{p}{3},\infty}$ in space-time, locally.
\end{corollary}

\begin{proof}
  Use the identity \eqref{D_decomp_NS} with the last two terms computed as in \eqref{reass_last_term} to bound 
  \begin{align}
       \abs{\left\langle   \mathcal{E}^\nu-\mathcal{E}^\nu* \rho_\delta,\varphi\right\rangle} &= \abs{ \left\langle   \mathcal{E}^\nu, \varphi - \varphi* \rho_\delta\right\rangle } \\
       &\lesssim \norm{E^{\ell,\nu}}_{L^\frac{p}{2}_{x,t}} \left( \norm{\partial_t (\varphi - \varphi * \rho_\delta)}_{L^{\frac{p}{p-2}}_{x,t}} + \norm{ u^\nu_\ell}_{L^p_{x,t}} \norm{\nabla (\varphi - \varphi * \rho_\delta)}_{L^{\frac{p}{p-3}}_{x,t}} \right)\\
       &\quad + \nu  \norm{E^{\ell,\nu}}_{L^\frac{p}{2}_{x,t}} \norm{\Delta (\varphi - \varphi * \rho_\delta)}_{L^{\frac{p}{p-2}}_{x,t}}\\
    &\quad +\norm{Q^{\ell,\nu}}_{L^\frac{p}{3}_{x,t}}\norm{\nabla (\varphi - \varphi * \rho_\delta)}_{L^{\frac{p}{p-3}}_{x,t}} + \norm{C^{\ell,\nu}}_{L^\frac{p}{3}_{x,t}}\norm{\varphi - \varphi * \rho_\delta}_{L^{\frac{p}{p-3}}_{x,t}}\\
    &\quad + \nu \left( \norm{\abs{\nabla u^\nu_\ell}^2}_{L^\frac{p}{2}_{x,t}} + \norm{u^\nu-u^\nu_\ell}_{L^p_{x,t}} \norm{\Delta u^\nu_\ell}_{L^p_{x,t}}\right) \norm{\varphi - \varphi * \rho_\delta}_{L^{\frac{p}{p-2}}_{x,t}}\\
    &\quad + \nu \norm{ u^\nu- u^\nu_\ell}_{L^p_{x,t}}\norm{\nabla u^\nu_\ell}_{L^p_{x,t}} \norm{\nabla (\varphi - \varphi * \rho_\delta)}_{L^{\frac{p}{p-2}}_{x,t}}.
  \end{align}
  Thus, by \eqref{moll_est_1}, \eqref{moll_est_2}, \eqref{est_Eell}, \eqref{est_Qell} and  \eqref{est_Cell} we bound
  \begin{align}
       \abs{\left\langle   \mathcal{E}^\nu-\mathcal{E}^\nu* \rho_\delta,\varphi\right\rangle } &\lesssim \left( \ell^{2\sigma}(\delta+\nu) + \ell^{3\sigma} \delta+\ell^{3\sigma-1} \delta^2 + \nu \ell^{2(\sigma-1)} \delta^2+ \nu \ell^{2\sigma-1} \delta \right) \norm{\varphi }_{W^{2,\frac{p}{p-3}}_{x,t}} \\
       &\lesssim (\ell^{2\sigma}+\ell^{2(\sigma-1)} \delta^2+\ell^{2\sigma-1} \delta) \norm{\varphi}_{W^{2,\frac{p}{p-3}}_{x,t}}\lesssim \delta^{2\sigma} \norm{\varphi }_{W^{2,\frac{p}{p-3}}_{x,t}},
  \end{align}
  where we have used that $\nu<1$ and in the last inequality we have chosen $\ell=\delta$. Similarly, one can estimate
  $$
\abs{\left\langle   \mathcal{E}^\nu* \rho_\delta,\varphi\right\rangle } \lesssim\delta^{2\sigma-2} \norm{\varphi}_{L^{\frac{p}{p-3}}_{x,t}}.
$$
The implicit constants in all the inequalities above depend only on local $L^p_tB^\sigma_{p,\infty}$ norms of $u^\nu$, which we are assuming to enjoy a bound uniform in viscosity. This proves \eqref{moll_est_dissipation_NS}. Then, the fact that $\left\{ \mathcal E^\nu\right\}_{\nu>0}$ stays bounded in $B^{2(\sigma-1)}_{\frac{p}{3},\infty}$ follows by the same argument of \textsc{Step 1} in the proof of Theorem \ref{T:main D regularity}.
\end{proof}


\section{Discussion} \label{s: comments} 

\subsection{Intermittency in turbulence}\label{S:intermittency turbulence}
In this section $u^\nu$ will denote a sufficiently regular solution to \eqref{NS} with $\nu>0$, in three dimensions. Although the result is true in any spatial dimension $d\geq 2$, the only physically meaningful case is when $d=3$. This is because in two dimensions the strong $L^2_{x,t}$ compactness is already inconsistent with a non-trivial dissipation \cites{LMP21,DRP25}. We introduce the ``absolute structure functions exponents" $\zeta_p$ as
$$
\langle |u^\nu(x+\ell z)-u^\nu(x)|^p\rangle \sim \ell^{\zeta_p} \qquad z\in \mathbb{S}^{2}, \, p\geq 1.
$$
Here the symbol $\langle\cdot\rangle$ denotes some relevant averaging procedure that might be space, time or ensemble. Mathematically, this translates into an exact $B^{\sigma_p}_{p,\infty}$ spatial Besov regularity with $\sigma_p=\frac{\zeta_p}{p}$.
Under the assumptions of homogeneity, isotropy, self-similarity and that all the main statistics of the fluid are completely determined by the non-trivial  energy dissipation rate, the Kolmogorov theory of turbulence \cite{K41} from 1941 predicts the universal dependence $\zeta_p=\frac{p}{3}$, linear in $p\in [1,\infty]$. This yields a $\frac13$ H\"older exponent of the velocity field uniform in viscosity, space and time, connecting the Kolmogorov statistical framework to the subsequent ideal and deterministic picture of Onsager \cite{O49} from 1949. We refer to the recent essay \cite{E24} describing the Onsager contributions to the theory of turbulence.

{Landau expressed his scepticism about Kolmogorov's hypotheses already in 1942  (see \cite{Frisch91} for a historical account). As later confirmed, there is indeed strong empirical evidence \cites{AGHA84,Sigg82,ISY20,BT49} that, in actual turbulence, space-time homogeneity and self-similarity break down, making any approach based on the latter inadequate to describe flows at high Reynolds number.  In their visionary work \cite{BT49}, among the first to detect intermittency,  Batchelor and Townsend write
 \begin{changemargin}{1cm}{1cm} 
\begin{center}
\textit{...the energy associated with large wave-numbers is very unevenly distributed in space. There appear to be isolated regions in which the large wave-numbers are ‘activated’, separated by regions of comparative quiescence.}
\end{center}
\end{changemargin}
 } 
 This ubiquitous phenomenon is known as ``intermittency" \cites{Frisch95,BT23}.  Despite several attempts \cites{K62,FP85,PV87,Frisch91,Mandelbrot,CJPV94,BPPV84,MBCBR95}, a quantitative theoretical understanding of intermittency from first principles is still missing.  What appears true from observation is the emergence of a continuous spectrum of H\"older exponents spreading over the domain, possibly resulting into a spotty\footnote{There is evidence for the velocity field to retain a H\"older exponent strictly bigger than $\frac13$ on most of the domain \cite{Sidd24}.} and (multi)fractal distribution of the energetically active regions in space-time \cites{Meneveau87,Meneveau88,Meneveau91}.  See  \cite{Falc}*{Chapter 17} for a rigorous analysis  of multifractal measures {and \cite{dubrulle18} for recent experimental and numerical studies on singularities, fractal dissipation and intermittency.} Along these lines, our results show a quantitative downward deviation from the exponent $\frac{1}{3}$ for all moments $p>3$. More precisely, the following is a direct consequence of Corollary \ref{C:main intermittency}.

 \begin{corollary}[Vanishing viscosity intermittency]
     On $\T^3\times (0,T)$, let $\{u^\nu\}_{\nu>0}\subset L^2_t H^1_x$ be a sequence of weak solutions to \eqref{NS}. Assume that, in the limit as $\nu\rightarrow 0$, $\mathcal{E}^\nu:=D^\nu+\nu |\nabla u^\nu|^2$ converges in $\mathcal D'_{x,t}$ to a measure whose singular part with respect to the Lebesgue measure is non-trivial and concentrated on a set $S$ with $\dim_{\mathcal H} S=\gamma \in [ 1,4]$.  For all  $p\in [3,\infty]$ for which there exists $\zeta_p\in (0,p)$ such that $\{u^\nu\}_{\nu>0}$ stays bounded in $ L^p_t B^{\frac{\zeta_p}{p}}_{p,\infty}$, it must hold
    \begin{equation}\label{eq: intermittenc viscos}
\zeta_p \leq \frac{p}{3} -  \frac{2\kappa(p-3)p}{ 9p -3\kappa (p-3)},
            \end{equation}
            where $\kappa:=4-\gamma$ is the codimension of the dissipation concentration set.
 \end{corollary}
Some remarks are in order. The assumption on the distributional convergence of the sequence $\{\mathcal E^\nu\}_{\nu>0}$ towards a measure is  satisfied in several cases, possibly up to subsequences. Indeed, for any sequence of suitable Leray--Hopf weak solutions emanating from $L^2_x$ bounded initial data,  $\{\mathcal{E}^\nu\}_{\nu>0}$ is a bounded sequence of positive space-time measures, and then it admits a weak limit in $\mathcal M_{x,t}$. The restriction to $\gamma \geq 1$, i.e. $\kappa\leq 3$, is natural since otherwise finding a $\sigma_p>0$ for which \eqref{eq:intermittency} holds would necessarily require $p$ to be quantitatively below $\frac{9}{2}$. This is unnatural in view of the Sobolev embedding $B^\frac13_{3,\infty}\subset L^{\frac92-}$  and the exactness of the Four-Fifths law. In this range of parameters, it is then clear that \eqref{eq: intermittenc viscos} forces the structure functions exponents to quantitatively deviate from the Kolmogorov prediction $\zeta_p=\frac{p}{3}$ for all $p>3$ as soon as $\kappa >0$ and the dissipation is a non-trivial lower-dimensional measure. Consequently, intermittency must happen and the larger the $p$  the larger the discrepancy. If $\kappa < 1$, there will be a value $p^*\in (3,\infty)$ at which the right hand side in \eqref{eq: intermittenc viscos} vanishes. In that case, the solution cannot have any fractional regularity in $L^{p}$ for all $p\geq p^*$. The results from \cite{DDI24} can be then applied, matching with the numerology. It is important to note that for $p=3$ lower dimensionality does not give any correction, which is  consistent with the Kolmogorov Four-Fifths law being exact.

The use of Hausdorff to measure the dimension  allows to deduce intermittency even if the concentration set {of the singular part} happens to be locally dense. In view of the wild behavior at high Reynolds numbers, this seems to be essential to capture relevant dissipative properties of the flow. Indeed, even simpler dynamical models such as the one-dimensional Burgers equation might exhibit shocks, and thus dissipation, proliferating over a dense set \cite{SAF92}. It is therefore reasonable to expect an even more intricate scenario for the Navier--Stokes equations, where the dynamics is further shaped by geometric constraints.

\subsection{Comparison with experimental data} The celebrated works of Meneveau and Sreenivasan studied the relationship between properties of the energy dissipation measure and intermittency 
\cites{Meneveau87,Meneveau88,Meneveau91}. These papers  suggest from experiments that, in the infinite Reynolds number limit, the anomalous energy dissipation measure at fixed time is concentrated on a fractal subset of dimension less than the space dimension 3, about 2.87, and has volume zero \cite{Meneveau87}*{Section 4}. Moreover, based on the data, it is supposed that this fractal dimension is roughly constant in time. We now give an argument for this value based on the formalism developed in this paper.
It seems reasonable that our considerations are very close to be sharp locally\footnote{We are mainly appealing to two facts for $p\approx 3$: the dissipation is what constraints the regularity the most and the almost saturation of our upper bound. Modulo the possible gap of longitudinal vs absolute increments, the exactness of the Four-Fifths law makes both claims valid at $p=3$.}  
at $p=3$.  Set
$$
\zeta_p^* := \frac{p}{3} -  \frac{2\kappa(p-3)p}{ 9p -3\kappa (p-3)},
$$
i.e. the right hand side in $\eqref{eq: intermittenc viscos}$.
In dimension $d=3$, it is readily verified that
$$
\tfrac{d \zeta_p^*}{dp}\Big|_{p=3}=\frac{3-2\kappa}{9}=\frac{2\gamma -5}{9}.
$$
High resolution direct numerical simulations of incompressible turbulence  \cite{iyer2020scaling} indicate $\frac{d \zeta_p}{dp}\Big|_{p=3} = 0.303   \pm 6.4\times 10^{-4}$.  This corresponds to $\gamma = 3.85$, remarkably close to the observations of  Meneveau and Sreenivasan as if the time evolution affects the estimated dimension in a purely additive way $2.87+1$. See \cref{fig1} for an inspection of numerical data and this bound.
		\begin{figure}
		\centering
			\includegraphics[width=0.4\textwidth]{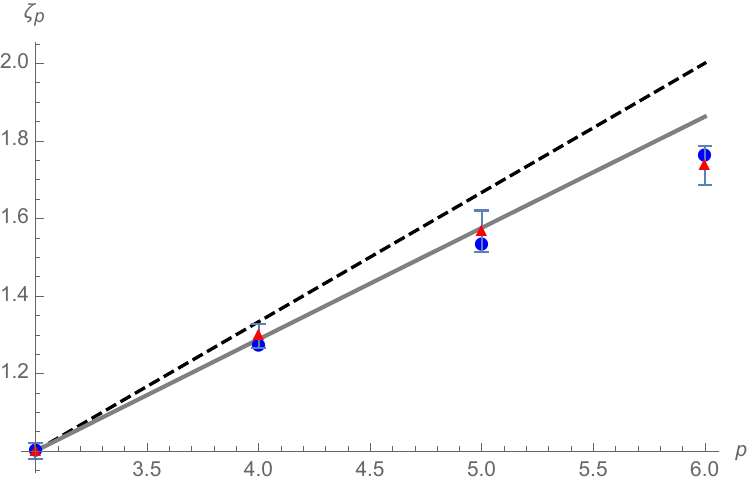} 
					\caption{Structure functions exponents for $p\in[3,6]$.  Blue dots are absolute structure functions exponents measured from the JHU turbulence database. Red triangles are transverse exponents reported in \cite{iyer2020scaling}.  Dashed black line corresponds to the Kolmogorov prediction of $\frac{p}{3}$.  Solid grey line corresponds to our bound $\zeta_p^*$ with $\gamma = 3.85$ inferred from  \cite{iyer2020scaling}.}\label{fig1}
	\end{figure}

 \subsection{Sharpness of the results \& convex integration} 
 Recent years have seen a quite intense mathematical work \cites{DS13,BDIS15,DanSz17,IO17,Buck15} in producing H\"older continuous Euler flows with non-constant kinetic energy, culminating in the resolution of the Onsager conjecture \cites{Is18,BDSV19}. See also \cites{Is24,DRS21,DT22} for refinements. H\"older continuous weak solutions with non-negative $D$ have also been produced \cites{DK22,Is22}, getting closer to the physical case in which Euler arises in the inviscid limit of Navier--Stokes. The method goes under the name of ``convex integration", introduced in this context by De Lellis and Sz\'ekelyhidi \cite{DS13}. The Onsager conjecture being true, at least arbitrarily close to the critical regularity \cite{Is24}, validates the Kolmogorov self-similar and isotropic prediction of a spatially homogeneous $\frac13$ H\"older exponent, although only in the ideal inviscid setting. Very recent works \cites{brue2022onsager,brue2022,SJ23} also prove anomalous dissipation in the inviscid limit, but with the use of an external forcing term. Remarkably, the convex integration methods have been recently modified to incorporate intermittency \cites{NV22,GKN24_1,GKN24_2,BMNV23}, producing Euler weak solutions  with a non-trivial dependence of $\zeta_p$ on $p$, thus going beyond H\"older regularity. 
 
 {All the above works construct dissipative solutions with supercritical regularity. Energy conservation in any subcritical class was previously established in \cites{Eyink95,CET94,CCFS08,bardos2019onsager} by commutator arguments \`a la DiPerna--Lions \cite{DipLi89}. See also the recent work \cite{DIN24} pointing out the relevance of the density of smooth functions. The only critical class where a full picture is currently available is the case of \quotes{organized singularities} \cites{Shv09,DIN24}, or more generally $L^1_tBV_x\cap L^\infty_{x,t}$  \cites{DRINV23, Inv26}. In these settings the incompressibility prohibits energy dissipation although the cubic energy flux may not vanish for kinematic reasons, highlighting the relevance of the geometric constraint $\div u=0$. What happens in general critical Onsager's classes, e.g. $L^\infty_tC^\frac13_x$, remains an important open problem.}

 Producing weak solutions of \eqref{E} with a non-trivial lower-dimensional measure $D$ is not yet done. However, the very recent intermittent constructions \cites{GKN24_1,GKN24_2,NV22,BMNV23} are going into this direction. What is certainly true about the optimality of our results is that the negative Besov regularity $D\in B^{\frac{2\sigma}{1-\sigma}-1}_{\frac{p}{3},\infty}$ cannot be improved in general.  Indeed we have the following. 
 
 \begin{theorem}\label{T:sharp_baire}
    For any $\sigma \in \left(0,\frac13 \right)$ and $p\in [3,\infty]$ there exist uncountably many weak solutions $u\in L^p_t B^\sigma_{p,\infty}$ to \eqref{E} in $d=3$ such that 
    $$
    D\in  B^{\frac{2\sigma}{1-\sigma}-1}_{\frac{p}{3},\infty}\setminus\bigcup_{\eps>0}B^{\frac{2\sigma}{1-\sigma}-1+\eps}_{\frac{p}{3},\infty}.
    $$
\end{theorem}
\begin{proof}
    Let $\sigma\in \left(0,\frac13\right)$. Pick any element $u$ in the residual set of \cite{DT22}*{Theorem 1.2}. Then $u\in L^\infty_t C^\sigma_x$ and, denoting by $e$ its kinetic energy, it holds 
    \begin{equation}
        \label{kin en holder sharp}
        e\in C^{\frac{2\sigma}{1-\sigma}}_t\setminus \bigcup_{\eps>0}W^{\frac{2\sigma}{1-\sigma}+\eps,1}_t.
    \end{equation}
   Clearly $u\in L^p_t B^\sigma_{p,\infty}$ for all $p\in [3,\infty]$. Thus, by Theorem \ref{T:main D regularity} we get $D\in B^{\frac{2\sigma}{1-\sigma}-1}_{\frac{p}{3},\infty}$ locally in space-time. To conclude, it is enough to prove $D\not\in  B^{-\alpha}_{1,\infty}$ for all $\alpha<1-\frac{2\sigma}{1-\sigma}$. By contradiction, assume $D\in B^{-\alpha}_{1,\infty}$ for some $\alpha<1-\frac{2\sigma}{1-\sigma}$. By possibly enlarging the value of $\alpha$, but keeping $\alpha<1-\frac{2\sigma}{1-\sigma}$, we have $D\in W^{-\alpha,s}_{x,t}$ for some $s>1$. Since we are in the spatially periodic setting, $D$ can be applied to space independent functions. Then
\begin{equation}\label{preparing regularity}
    \left|\langle  e' ,\eta\rangle \right|=\left|\langle D,\eta\rangle \right|\lesssim \|\eta\|_{W^{\alpha,s'}_t}\qquad \forall \eta\in C^\infty_t,
    \end{equation}
that is $e'\in W^{-\alpha,s}_t$.
It follows that\footnote{This can be easily seen by the Bernstein-type inequality $\|e_k\|_{L^s_t}\lesssim 2^{-k}\|e'_k\|_{L^s_t}$ for all $k\geq 2$, where $e_k:=e*\phi_k$ denotes the frequency localization (see for instance \cite{BCD11}*{Lemma 2.1}).} $e\in  W^{1-\alpha,s}_t $, contradicting \eqref{kin en holder sharp}.
\end{proof}
   Note that, although the time marginal could be prescribed\footnote{By choosing in \cite{DT22}*{Theorem 1.1} a monotone non-increasing Cantor function with a derivative concentrated on a given set, it is possible to construct solutions whose time marginal of $D$ is a measure supported on a set of times with Hausdorff/Minkowski dimension arbitrarily close to $\frac{2\sigma}{1-\sigma}$. }, for these solutions $D$ is only a space-time distribution and not necessarily a measure. Moreover, in view of the mechanism behind the H\"older based convex integration scheme, there is no a priori control over the space-time support of $D$. Another limitation of \cref{T:sharp_baire} is that the space-time saturation of the regularity of $D$ is fully dictated by that in the time variable, which is unlikely to be realistic. With that being said, it would be interesting to address the sharpness of \cref{T:main D regularity}, possibly for a non-negative dissipation $D$ exhibiting lower dimensionality in space-time. In this context, the stronger statement would be to prescribe any $D$, picked in a suitable subset of $B^{\frac{2\sigma}{1-\sigma}-1}_{\frac{p}{3},\infty}$, for solutions in the corresponding Besov class, perhaps showing full flexibility for Euler flows in saturating any geometrical/analytical constraint imposed by the PDE. The work \cite{DHaft22} is closely related to this discussion,  imposing the additional constraint on the solution being smooth outside the closed dissipative set of times. Differently from \cite{DT22}, the construction in \cite{DHaft22} also provides a dissipation with a space-time lower-dimensional support, although a bit far from the presumed sharp threshold.

An interesting feature of the regularity $D\in B^{\frac{2\sigma}{1-\sigma}-1}_{\frac{p}{3},\infty}$ is that it does not seem to be easily achievable by showing that an approximation of $D$ stays bounded in that space. For instance, although a natural candidate would be the viscous dissipation $\mathcal E^\nu=D^\nu +\nu |\nabla u^\nu|^2$, it is not clear to the authors how to improve the uniform bound of Corollary \ref{C:unif_visc_bound}. The same seems to happen with both the Duchon--Robert approximation \cite{DR00} and the Constantin--E--Titi one \cite{CET94}, somehow suggesting that all the available approximations do not capture essential cancellations which however appear in the limiting object. If the bound in Corollary \ref{C:unif_visc_bound} cannot be improved to the optimal one, there might be fine dissipative mechanisms arising in the limit $\nu\rightarrow 0$ that stand apart from their measurements at very high, but finite, Reynolds numbers.

On a different side than constructing solutions, there are some other recent works addressing the issue of intermittency. For instance, an extensive ``volumetric approach" has been developed in \cites{CS1,CS23} to extract information from the most energetically active parts of the flow at a given scale. This allows to analytically define a notion of dimension and, among other things, to validate the Frisch--Parisi multifractal formalism \cite{FP85}. See also \cites{BS23I,BS23II,J00}. More related to the spirit of this paper are \cites{DRIS24,DDI24}. In particular, in \cite{DRIS24} intermittency for Besov solutions is deduced assuming lower dimensionality of the dissipation in the Minkowski sense, while in \cite{DDI24} by means of the Hausdorff dimension but only for integrable weak solutions, thus failing in making any non-trivial use of fractional regularity. The current paper closes the gap and reconciles the two approaches. 

\subsection{General open sets} \label{S:open_set}
Although all the results in this paper are stated in the spatially periodic setting $\T^d$, we emphasize that they are all intrinsically local. Thus, they carry over any open set $\Omega\subset \R^d$, of course away from the boundary. To do that rigorously, the only thing that has to be fixed is the usual issue with the pressure. The latter being determined only up to arbitrary time dependent functions might be in general not enough to deduce the double pressure regularity, even if only in the interior. As shown in the lemma below, it is enough that the spatial average of $|q|$ has a suitable time integrability. 

\begin{lemma}\label{L:pressure}
Let $\Omega\subset \R^d$ be open. Assume 
$$
-\Delta q =\div \div (u\otimes u) \qquad \text{ in  } \Omega\times (0,T).
$$
Let $p\in (2,\infty]$ and $\sigma\in \left(0,\frac12\right)$. If $q\in L^\frac{p}{2}_{t}L^1_x$ and $u\in L^{p}_t B^\sigma_{p,\infty}$ locally, then $q\in L^\frac{p}{2}_t B^{2\sigma}_{\frac{p}{2},\infty}$ locally.
\end{lemma}
\begin{proof}
   Let $U\subset \joinrel \subset \Omega$ be any open set. In what follows $t$ is any fixed instant of time, picked in a full measure subset of $(0,T)$. Since $u(t)\in B^\sigma_{p,\infty}$ locally inside $\Omega$, we find a divergence-free and compactly supported $\tilde u(t)\in B^\sigma_{p,\infty}$ on $\R^d$ such that $\tilde u(t) \equiv u(t)$ on $U$.  Let $\tilde q$ be the unique solution to 
   $$
   -\Delta \tilde q =\div \div (\tilde u\otimes \tilde u) \qquad \text{ in  } \R^d
   $$
   decaying at infinity. By \cites{CDF20,Isett23,ColDeRos} and the continuity of the extension operator
   $$
   \| \tilde q(t)\|_{B^{2\sigma}_{\frac{p}{2},\infty}}\lesssim   \| \tilde u(t)\|^2_{B^\sigma_{p,\infty}}\lesssim   \| u(t)\|^2_{B^\sigma_{p,\infty}(U)},
   $$
   from which we deduce $\tilde q \in L^\frac{p}{2}_t B^{2\sigma}_{\frac{p}{2},\infty}$. Since $q=q-\tilde q +\tilde q$, we are only left to prove $q-\tilde q \in L^\frac{p}{2}_t B^{2\sigma}_{\frac{p}{2},\infty}$ locally inside $U\times (0,T)$. However, since $(q-\tilde q)(t)$ is harmonic in  $U$, the mean value property implies that  its $C^1_x$ norm on compact subsets of $U$ is bounded by the $L^1_x$ norm on $U$. Since  $q \in L^\frac{p}{2}_{t}L^1_x$ by assumption, the proof is concluded.
\end{proof}
Once double pressure regularity holds, all the results in this paper can be replicated in the interior of any open set $\Omega$. 
 \subsection{Intermittency in scalar turbulence} \label{transport}
Let $\Omega\subset \R^d$ be open. Given an incompressible vector field $v:\Omega\times (0,T)\rightarrow \R^d$, consider the transport equation
\begin{equation}\label{T} \tag{T}
\partial_t \theta +\div (\theta v )  =0
 \qquad \text{in }\Omega \times (0,T).
\end{equation}  
The local dissipation $\tilde D$ can be defined as
$$
\partial_t \frac{\abs{\theta}^2}{2} +\div \left(\frac{\abs{\theta}^2}{2} v \right)=-\tilde D \qquad \text{in } \mathcal D'_{x,t}
$$
as soon as $v\in L^p_{x,t}$, $\theta\in L^s_{x,t}$ with $\frac{1}{p}+\frac{2}{s}\leq 1$.  Note that, being a linear equation, weak solutions can be obtained from weak compactness of vanishing diffusivity approximations. Provided that the limit is achieved strongly, it follows that $\tilde D$ is actually a non-negative Radon measure.

Passive scalar transport is a well studied physical system, and there is a wealth of evidence for anomalous dissipation therein \cites{donzis2005scalar,sreenivasan2019turbulent}. Obukhov \cites{obukhov1949structure} and Corrsin \cites{corrsin1951spectrum} derived Onsager-type predictions on the requisite degree of singularity required to see anomalous dissipation in this context. The result is, roughly, that if $\sigma$ represents the fractional regularity of the velocity, the scalar cannot have regularity $\beta$ greater than $\frac{1-\sigma}{2}$. This was made rigorous by Eyink \cite{eyink1996intermittency}, following the works of  Constantin and Procaccia \cites{constantin1993scaling,constantin1994geometry}.  See also the discussion in \cite{drivas2022anomalous}. Recently, there have been mathematical constructions of passive scalars exhibiting anomalous dissipation \cites{drivas2022anomalous,armstrong2025anomalous,BSW23,hess2025universal,SJ24}, even some that nearly exhibit the sharpness of the Obukhov--Corrsin theory in H\"{o}lder spaces \cites{colombo2023anomalous,elgindi2024norm,BSW26}. The end point cases remain open.  In physical situations, however, it is widely expected that the scalar is far from monofractal, and displays anomalous scaling exponents, e.g. the Obukhov--Corrsin theory in H\"{o}lder spaces is far from sharp and instead it holds only in an appropriate critical class: $L^2_x$ on the scalar rather than $L^\infty_x$ if the velocity is H\"{o}lder continuous.  This expectation has a great deal of numerical justification \cites{iyer2018steep,sreenivasan2019turbulent} and a theoretical one in the  Kraichnan model of passive scalar turbulence \cites{bernard1996anomalous,gawedzki1998intermittency}.

The analysis performed here for the Euler equations can be extended to the setting of passive scalars and gives precise constraints on the relation between intermittency and dissipation measure.  In the following theorem we highlight the main implications of our framework.

\begin{theorem}\label{T:passive scalar}
    Let $v\in L^p_{x,t}$ be a given vector field for some $p\in [1,\infty]$. Let $\theta\in L^s_{x,t}$ be a weak solution to \eqref{T} with $\frac{1}{p}+\frac{2}{s}\leq 1$ and with local dissipation $\tilde D\in \mathcal D'_{x,t}$. Set
    \begin{align}
     \tilde E^\ell:=\frac{\abs{\theta-\theta_\ell}^2}{2}, \qquad &\tilde Q^\ell :=\frac{\abs{\theta-\theta_\ell}^2}{2} (v-v_\ell), \qquad \tilde R^\ell:= \theta_\ell v_\ell - (\theta v)_\ell\\ 
     \text{ and } \qquad \tilde C^\ell &:=(\theta - \theta_\ell)\left((v-v_\ell)\cdot \nabla \theta_\ell +\div \tilde R^\ell \right),
     \end{align}
     where the subscripts $\ell$ denote the space mollification. For all $\ell>0$,  the following identity holds
     \begin{equation}
         \label{D_decomp_transp}
         -\tilde D= (\partial_t+v_\ell \cdot \nabla) \tilde E^\ell +  \div \tilde Q^\ell + \tilde C^\ell \qquad \text{in } \mathcal D'_{x,t}.
     \end{equation}
     Assume that $v\in L^p_t B^\sigma_{p,\infty}$ and $\theta\in L^s_t B^\beta_{s,\infty}$ for some $\sigma,\beta\in (0,1)$. Then $\tilde D\in B^{\frac{2\beta}{1-\sigma}-1}_{\frac{ps}{2p+s},\infty}$ locally. If in addition $\tilde D$ is a real-valued Radon measure, we have $|\tilde D|\ll \mathcal H^\gamma$ for any $\gamma \geq 0$ such that 
            \begin{equation}
                \frac{2\beta}{1-\sigma}>1- \frac{p(s-2)-s}{ps} (d+1-\gamma).
            \end{equation}
            When $\tilde D\geq 0$, $\forall K$ compact $\exists r_0>0$ such that
            $$
            \tilde D (B_r(x,t))\lesssim r^{ \frac{2\beta}{1-\sigma}-1+ \frac{p(s-2)-s}{ps} (d+1)}\qquad \forall (x,t)\in K, \, \forall r<r_0.
            $$
\end{theorem}
Note that $\abs{\theta}^2 v \in L^{\frac{ps}{2p+s}}_{x,t}$ and its H\"older dual is precisely $\frac{ps}{p(s-2)-s}=1-\frac{2}{s}-\frac{1}{p}$, making clear their appearance in the theorem above. 
An immediate implication of Theorem \ref{T:passive scalar} is the following intermittency-type statement. Since the arguments are the same we have given for Euler, we will omit the proof.

\begin{corollary}
 Let $v\in L^p_{x,t}$ be a given vector field for some $p\in [1,\infty]$. Let $\theta\in L^s_{x,t}$ be a weak solution to \eqref{T} with $\frac{1}{p}+\frac{2}{s}\leq 1$ and with a  local dissipation measure $\tilde D$ with a non-trivial singular part concentrated on a space-time set $S$ with $\dim_{\mathcal H} S=\gamma$. For all such $p$ and $s$  for which there exist $\sigma_p, \beta_{s} \in (0,1)$ such that $v\in L^p_t B^{\sigma_p}_{p,\infty}$ and $\theta\in L^s_t B^{\beta_s}_{s,\infty}$, it must hold 
            $$
            \frac{2\beta_s}{1-\sigma_p}\leq 1- \frac{p(s-2)-s}{ps} (d+1-\gamma).
            $$
\end{corollary}
This result immediately recovers the Obukhov--Corrsin and Eyink bounds, which say that non-trivial dissipation requires  $\sigma_p + 2 \beta_s\leq 1$ for all $p\in [1,\infty]$, and represents a substantial refinement in case more is known about the set where the dissipation measure concentrates.  For instance, numerical work on scalar advection by three-dimensional turbulent velocity fields \cite{iyer2018steep} suggests that exponents saturate $s\beta_s \to 1.2$ as $s\rightarrow \infty$, approximately, thus quite far from being monofractal. We believe that constructions in the spirit of  \cites{drivas2022anomalous,armstrong2025anomalous,colombo2023anomalous,elgindi2024norm,BSW23,hess2025universal,BSW26} could be made to show the sharpness of this intermittent Obukhov--Corrsin theory.

\textbf{Acknowledgments:} We thank Kartik Iyer for supplying us with the structure functions data from \cite{iyer2020scaling}. The research of TDD was partially supported by the NSF DMS-2106233 grant, NSF CAREER award \# 2235395, a Stony Brook University Trustee’s award as well as an Alfred P. Sloan Fellowship.  The research of PI is supported by the NSF grant DMS 2346799 and an Alfred P. Sloan Fellowship. LDR and MI acknowledge the support of the SNF grant Fluids, Turbulence, Advection (FLUTURA) \# 212573.

\textbf{Data Availability Statement:}  There is no unpublished data associated to this study.

\bibliographystyle{plain} 
\bibliography{biblio}

\end{document}